%% file: main.tex
\documentclass{article}
\input{preamble}

\begin{document}

\maketitle

\begin{abstract}
We construct global versions of the analytic Hopf algebras used in the $p$-adic Fourier theory of Amice and Mahler over a general Banach ring, independently of the choice of prime $p$. This is done by generalising K\"othe echelon and coechelon spaces to an arbitrary base Banach ring $R$ and proving reflexivity and nuclearity results. We show how to define an analytic Hopf algebra structure on them and investigate their duality theory. The particular case of the Hopf algebra of analytic functions converging on the open unit disk around $1$ and its dual is studied in detail. Amice duality is recovered from this case by base-change to a $p$-adic ring. Most notably, when $R$ is the ring of integers with the trivial norm, we obtain a global analytic version of Amice duality that does not depend on $p$.
\end{abstract}

\tableofcontents
\pagebreak

\section*{Introduction}
\addcontentsline{toc}{section}{Introduction}
\input{sections/intro}

\subsection*{Acknowledgements}
Part of this work is based on unpublished notes by Federico Bambozzi, Kobi Kremnizer, and Adam Topaz. I thank them for sharing their insights. I am especially grateful to Federico Bambozzi for his constant support throughout the writing of this article.
\pagebreak

\subsection*{Notation and conventions}
\input{sections/conv}

\section{Recalls on bornological modules} \label{sec:born}
\input{sections/born}

\section{A family of reflexive bornological modules} \label{sec:reflex}
\input{sections/reflex}

\section{Duality of analytic Hopf algebras} \label{sec:duality}
\input{sections/duality}

\subsection{Application: uniform \textit{p}-adic Fourier duality}
\input{sections/concl}

\printbibliography[heading=bibintoc]

\end{document}

%% file: preamble.tex
\usepackage[utf8]{inputenc}

\usepackage{amsmath, amsfonts, amsthm, graphicx, geometry, enumerate}
\numberwithin{equation}{section}

\usepackage{comment}

\usepackage{commath}

\usepackage{hyperref}
\hypersetup{
    colorlinks=true,
    linkcolor=black,
    citecolor=black,
    urlcolor=blue,
}
\usepackage{cleveref}

\usepackage{xcolor}

\usepackage{tikz}
\usetikzlibrary{cd}

\usepackage{dsfont}
\usepackage{mathrsfs}

\usepackage[
    backend=biber,
    style=alphabetic
 ]{biblatex}
\newtheorem{theorem}{Theorem}[section]
\newtheorem{prop}[theorem]{Proposition}
\newtheorem{lemma}[theorem]{Lemma}
\newtheorem{corollary}[theorem]{Corollary}

\newtheorem*{theorem*}{Theorem}

\theoremstyle{remark}
\newtheorem{remark}[theorem]{Remark}

\theoremstyle{definition}
\newtheorem{definition}[theorem]{Definition}

\title{Duality of analytic Hopf algebras and the Amice transform}
\author{Luca Collauto}
\date{}

\newcommand{\N}{\mathbb{N}}
\newcommand{\Np}{\mathbb{N}_{>0}}
\newcommand{\Z}{\mathbb{Z}}
\newcommand{\Q}{\mathbb{Q}}
\newcommand{\R}{\mathbb{R}}
\newcommand{\Rnn}{\mathbb{R}_{\geq0}}
\newcommand{\Rp}{\mathbb{R}_{>0}}
\newcommand{\C}{\mathbb{C}}

\newcommand{\A}{\mathbb{A}}

\newcommand{\na}{\mathrm{na}}  
\newcommand{\la}{\mathrm{la}}  
\newcommand{\op}{\mathrm{op}}  

\newcommand{\sA}{\mathscr{A}}
\newcommand{\sC}{\mathscr{C}}
\newcommand{\sD}{\mathscr{D}}

\newcommand{\sS}{\mathscr{S}}

\newcommand{\bC}{\mathbf{C}}

\newcommand{\Set}{\mathbf{Set}}

\newcommand{\Alg}{\mathbf{Alg}}
\newcommand{\Ind}{\mathbf{Ind}}
\newcommand{\Ban}{\mathbf{Ban}}
\newcommand{\Banc}{\mathbf{Ban}^{\leq1}}
\newcommand{\Bannac}{\mathbf{Ban}^{\na,\leq1}}
\newcommand{\Born}{\mathbf{Born}}
\newcommand{\Comm}{\mathbf{Comm}}
\newcommand{\Aff}{\mathbf{Aff}}

\newcommand{\fdiff}[2][1]{{#2}^{[#1]}}
\newcommand{\id}{\mathrm{id}}

\DeclareMathOperator{\Int}{Int}

\DeclareMathOperator{\Hom}{Hom}

\DeclareMathOperator{\End}{End}

\DeclareMathOperator*{\colim}{colim}
\DeclareMathOperator*\indlim{``\underrightarrow{\lim}''}
\DeclareMathOperator{\Ker}{Ker}
\DeclareMathOperator{\Coker}{Coker}
\DeclareMathOperator{\Coim}{Coim}
\DeclareMathOperator{\im}{Im}

\newcommand{\wot}{\widehat{\otimes}}
\newcommand{\ihom}{\underline{\operatorname{Hom}}}

\newcommand{\prodc}{\sideset{}{^{\leq1}}\prod}
\newcommand{\coprodc}{\sideset{}{^{\leq1}}\coprod}

\newcommand{\coprodnac}{\sideset{}{^{\na,\leq1}}\coprod}

\newcommand{\bs}[3]{\Int_{#1}\langle #2 \rangle^{\leq #3}}
\newcommand{\dbs}[3]{\Int_{#1}\langle #2 \rangle^{\dagger,#3}}

\newcommand{\ps}[3]{#1\langle #2 \rangle^{\leq #3}}

\newcommand{\sps}[3]{#1\langle #2 \rangle^{< #3}}


%% file: sections/intro.tex

In the development of arithmetic geometry, functional analysis has assumed a progressively more significant role, especially in the study of $L$-functions. Starting from Tate's thesis, the perspective on $L$-functions has adopted analytic tools like Fourier theory, spaces of functions, measures, and distributions. 

To get a picture, take Tate's construction of the Riemann zeta function and Dirichlet $L$-functions. Consider the ring of adèles $\A$ over the rational numbers, the space of rapidly decreasing complex valued functions $\sS(\A,\C)$ over it, and its dual, the space of tempered distributions $\sS'(\A,\C)$. The group of continuous homomorphisms from the quotient group $\A^\times /\Q^\times$ to $\C^\times$ is a complex variety $X(\C)$ parametrised by a complex number $s\in\C$ and Dirichlet character $\chi$. Tate's construction of the $\zeta$-integral provides an analytic map
\[
X(\C) \to \sS'(\A,\C),\quad (s,\chi)\mapsto \zeta(s,\chi).
\]
Then, for some particular choice of $f\in\sS(\A,\C)$, the duality pairing between $\sS'(\A,\C)$ and $\sS(\A,\C)$ induces a $\C$-analytic function on $X(\C)$
\[
\Lambda: X(\C) \to \C
\]
which coincides with the Dirichlet $L$-function $L(s,\chi)$ up to a $\Gamma$-factor $\Gamma_{\!\infty}(s,\chi)$. This construction allows one to obtain the analytic continuation and the functional equation of $L(s,\chi)$ by means of Fourier duality theory on $\sS(\A,\C)$ and $\sS'(\A,\C)$.

An idea that goes back to Iwasawa is to consider an analogous setting but with coefficients in $\C_p$, the completed algebraic closure of $\Q_p$, instead of $\C$. Similarly to above, one obtains a $\C_p$-analytic function $L_p(s,\chi)$ defined on a rigid-analytic variety over $\C_p$. The so called $p$-adic $L$-function $L_p(s,\chi)$ interpolates the critical values of the classical Dirichlet $L$-function associated with the Dirichlet character $\chi$.

The functional spaces in this setting are the space $\sC^\la(\Z_p,\C_p)$ of locally analytic functions on $\Z_p$ with values in $\C_p$ and its dual $\sD^\la(\Z_p,\C_p)$. 
The duality between the two spaces was analysed by Amice \cite{duals,inter}, who identified $\sD^\la(\Z_p,\C_p)$ with the space $\sps{\C_p}{T}{1}$ of analytic functions on the open unit disk in $\C_p$. The identification is constructed through the Amice transform: 
\[
\sD^\la(\Z_p,\C_p) \longrightarrow \sps{\C_p}{T}{1}, 
\quad
\mu \longmapsto \sA_\mu(T):=\int_{\Z_p} (1+T)^x\,\mu(x).
\]

The Amice transform also identifies the Hopf algebra structure induced on $\sD^\la(\Z_p,\C_p)$ by the addition law of the $p$-adic integers, with the Hopf algebra structure induced by the multiplication law defined on the unit disk with centre in $1$. The same remains true even if one changes coefficients and takes general complete subrings of $\C_p$. From a geometric point of view, the Amice duality underlies a duality between commutative analytic groups defined over $\Z_p$.

It is common in number theory to blend geometric and analytic tools and to apply them in parallel in both the $p$-adic world and the archimedean setting. Just like algebraic geometry incorporates commutative algebra as its local theory, analytic geometry seeks to develop a geometry whose local theory is based on functional spaces.

In recent years, the foundations of analytic geometry have advanced significantly. An example is the framework developed using condensed mathematics \cite{condmath}. Another approach is presented in \cite{persp} and \cite{dag,fremod,sheafy}. In these works, the local theory is formulated using a bornological approach to functional spaces. This means that the focus is on norms and bounded maps between spaces instead of open subsets and continuous linear maps. The bornological theory has already proven its relevance in the classical context over $\R$ and $\C$, and also over a non-archimedean field. The authors of \cite{dag,fremod,sheafy} managed to extend the theory in order to make it work on a general complete normed ring, also called a Banach ring. One of the most interesting features of this is the possibility of globalising analytic geometry. For example, when the ring is $\Z$ with the trivial or archimedean norm, one can perform the common constructions over it and then recover the same construction over $\Q_p$ or $\R$ with the correct analytic structure through base-change.

The goal of this article is to investigate the duality theory of analytic Hopf algebras in this bornological context over a general Banach ring, with a main focus on globalising Amice duality. 

In Section \ref{sec:born}, we recall the background on Banach, ind-Banach, and (complete) bornological modules over a Banach ring $(R,\abs{-})$. In particular, we introduce the dual of a module $M$ as in classical functional analysis: it is the module consisting of bounded linear maps $M \to R$. We also review the nuclearity property. It is a property relevant in functional analysis and analytic geometry, just as finiteness conditions on modules are for algebraic geometry (see, for example, Section 6 of \cite{kelly2025} and Lecture VIII of \cite{claused2026}).

In Section \ref{sec:reflex} we define a family of bornological modules that generalise to a Banach ring $(R,\abs{-})$ the spaces of sequences introduced by K\"othe \cite{koethe} for $\R$ or $\C$. Given a family $\rho=\{\rho_j:j\in\N\}$ of functions $\rho_j:\N \to \Rp$, we consider the module $\lambda(\rho)$ consisting of sequences $(x_0,x_1,x_2,\dots)$ of elements of the Banach ring $R$ such that the sum
\[
\abs{x_0}\rho_j(0) \;+\; \abs{x_1}\rho_j(1) \;+\; \abs{x_2}\rho_j(2) \;+\; \cdots
\]
converges for every $j$. The dual notion is the module $\kappa(\rho)$ consisting of sequences $(x_0,x_1,x_2,\dots)$ such that 
\[
\frac{\abs{x_0}}{\rho_j(0)} \,+\, \frac{\abs{x_1}}{\rho_j(1)} \,+\, \frac{\abs{x_2}}{\rho_j(2)} \,+\, \cdots
\]
converges for some $j$. The main result of this section is 

\begin{theorem*}[See Theorems \ref{th:duality_lambda-kappa} and \ref{th:nuclearity_kappa-lambda}]
Suppose that for every $j\in\N$ the sum 
\[
\sum_{n=0}^\infty \frac{\rho_j(n)}{\rho_{j+1}(n)}
\]
converges. Then $\lambda(\rho)$ and $\kappa(\rho)$ are nuclear and reflexive. Moreover, there is a bounded bilinear map 
\[
\lambda(\rho) \;\wot_R\; \kappa(\rho) \to R
\]
inducing an isomorphism between $\kappa(\rho)$ (resp. $\lambda(\rho)$) and the dual of $\lambda(\rho)$ (resp. $\kappa(\rho)$).
\end{theorem*}
The above theorem generalises the results of \cite{fremod}, in which it is proven that the algebra $\sps{R}{T}{r}$ of power series converging in the open disk of radius $r$ is nuclear (see \cite[Corollary 6.9]{fremod}). Indeed, $\lambda(\rho)$ is isomorphic to $\sps{R}{T}{r}$ when $\rho$ is defined by
\begin{gather*}
\rho_j(n)=\left( r-\frac{1}{j+1} \right)^n, \quad \text{or}\\
\rho_j(n)= (j+1)^n \quad \text{ if $r=\infty$,}
\end{gather*}
for all $j,n\in\N$. With other choices of weights $\rho_j$ and monoid structures on $\lambda(\rho)$ and $\kappa(\rho)$, one obtains global versions of many classical spaces of functions from archimedean and non-archimedean functional analysis. For example, overconvergent power series, smooth or locally analytic $p$-adic functions on $\Z_p$, Dirichlet or divided power series with convergence conditions, and so on. In the same section, we prove the same results about duality and nuclearity for countable products and direct sums of copies of $R$, as they include algebras like the algebra of formal power series or the polynomial algebra (see Corollary \ref{cor:duality_prod_coprod}).

In Section \ref{sec:duality}, we introduce (commutative) bialgebras and Hopf algebras in the bornological category, and we investigate their duality. Then we focus on the main application, which is the globalisation of Amice duality. For any non-archimedean Banach ring $R$, we define the two Hopf algebras $\sps{R}{s}{1}$ and $\dbs{R}{x}{1}$, and we prove 

\begin{theorem*}[See Theorem \ref{th:duality_of_the_bialgebras}]
There is a pairing
\[
\langle-,-\rangle:\sps{R}{s}{1} \;\wot_R\; \dbs{R}{x}{1} \to R
\]
under which $\sps{R}{s}{1}$ (resp. $\dbs{R}{x}{1}$) is the dual Hopf algebra of $\dbs{R}{x}{1}$ (resp. $\sps{R}{s}{1}$).
\end{theorem*}
For $R=\C_p$, one has
\begin{equation*}
\dbs{R}{x}{1} = \sC^\la(\Z_p,\C_p),
\end{equation*}
and the pairing $\langle-,-\rangle$ induces the Amice transform. The same pairing also induces an isomorphism identifying the dual of $\sps{\C_p}{s}{1}$ with $\sC^\la(\Z_p,\C_p)$, which is not considered in the classical references. The same remains true if $R$ is any complete subring of $\C_p$. A more interesting case is when $R$ is the ring of integers with the trivial norm, where the Hopf algebras $\sps{\Z}{s}{1}$ and $ \dbs{\Z}{x}{1}$, with their duality, constitute a global version of Amice duality: for every prime $p$, 
\begin{align*}
\dbs{\Z}{x}{1} \,\wot_\Z\, \Z_p &= \sC^\la(\Z_p,\Z_p), \\
\sps{\Z}{s}{1} \,\wot_\Z\, \Z_p &= \sps{\Z_p}{s}{1},
\end{align*}
and the base-change $-\wot_\Z \Z_p$ sends the pairing over $\Z$ to the Amice duality pairing over $\Z_p$.

%% file: sections/conv.tex
We use the following notation and conventions.
\begin{itemize}
    \item [-] number systems are denoted with blackboard bold symbols $\N,\Z,\Q,\R,\C,\Z_p,\Q_p,\C_p,\dots$;
    \item [-] $\N$ is the set of natural numbers with zero included;
    \item [-] $\Np$ is the set of non-zero natural numbers;
    \item [-] $\Rnn$ is the set of positive real numbers, including zero;
    \item [-] $\Rp$ is the set of non-zero positive real numbers;
    \item [-] rings are always assumed to be commutative with identity unless otherwise stated;
    \item [-] homomorphisms of rings preserve the identity;
    \item [-] for any ring $R$, the set of invertible elements in $R$ is denoted by $R^\times$.
\end{itemize}
We use the categorical language, and we adopt the usual convention on the existence of uncountable strongly inaccessible cardinals to avoid set-theoretic issues. If $\bC$ is a category, $X\in\bC$ means that $X$ is an object of $\bC$. Sets of homomorphisms are denoted by $\Hom(-,-)$. We use
\begin{itemize}
    \item [-] $\varinjlim$ for colimits or inductive limits;
    \item [-] $\varprojlim$ for limits or projective limits;
    \item [-] $\coprod$ for coproducts;
    \item [-] $\prod$ for products.
\end{itemize}

%% file: sections/born.tex
In this section, we recall the basic category in which we work throughout the paper. The main references are \cite{sheafy,dag,fremod}.

All the constructions that we will perform in the article are relative to a Banach ring $R$: this is a commutative ring with identity $R$ equipped with a function 
\[
\abs{-}: R \to \Rnn
\]
subject to the following conditions
\begin{enumerate}[i)]
    \item $\abs{-}$ is a \emph{norm}: for every $a,b\in R$ 
    \[
    \abs{a+b}\leq \abs{a}+\abs{b}
    \]
    and $\abs{a}=0$ if and only if $a=0$;
    \item $\abs{-}$ is \emph{sub-multiplicative}: for every $a,b\in R$ 
    \[
    \abs{ab}\leq\abs{a}\cdot\abs{b}
    \]
    and $\abs{1_R}\leq1$;
    \item $R$ is \emph{complete} with respect to the metric 
    \[
    R\times R \to \Rnn,
    \quad
    (a,b)\mapsto\abs{b-a}.
    \]
\end{enumerate}
Relative to the Banach ring $R$ there is a category $\Ban_R$ of \emph{Banach $R$-modules} whose objects are $R$-modules $M$ with a norm $\norm{-}:M\to\Rnn$ satisfying the following conditions
\begin{enumerate}[i)]
    \item $\norm{-}$ is a \emph{norm}: for every $x,y\in M$ 
    \[
    \norm{x+y}\leq \norm{x}+\norm{y}
    \]
    and $\norm{x}=0$ if and only if $x=0$;
    \item $\norm{-}$ is compatible with \emph{scalar multiplication}: for every $a\in R$ and $x\in M$
    \[
    \norm{a\cdot x}\leq\abs{a}\cdot\norm{x};
    \]
    \item $M$ is \emph{complete} with respect to the metric 
    \[
    M\times M \to \Rnn,
    \quad
    (x,y)\mapsto\norm{y-x}.
    \]
\end{enumerate}
Morphisms in $\Ban_R$ are \emph{bounded} $R$-linear maps $f:M \to N$ between Banach modules $(M,\norm{-}_M)$ and $(N,\norm{-}_N)$. Bounded means that there is a real positive constant $C$ such that
\[
\norm{f(x)}_N\leq C\cdot \norm{x}_M
\]
for every $x\in M$. The smallest such constant is denoted by $\norm{f}$ and is called the \emph{operator norm} of $f$. The operator norm can equivalently be expressed as the supremum
\[
\sup_{x\in M\setminus\{0\}} \frac{\norm{f(x)}_N}{\norm{x}_M}.
\]

\begin{definition}
    $\Banc_R$ is the non-full sub-category of $\Ban_R$ consisting of the same objects and \emph{contracting} morphisms, namely morphisms $f:M\to N$ such that $\norm{f}\leq1$. If a morphism satisfies $\norm{f(m)}_N=\norm{m}_M$ for every $m\in M$ it is said \emph{isometric}. For every real number $\rho>0$ and Banach $R$-module $M$, denote by 
    \[
    [M]_\rho
    \]
    the Banach $R$-modules with the same underlying $R$-modules and with the norm rescaled by $\rho$, i.e
    \[
    \norm{x}_{[M]_\rho}=\rho\norm{x}_M
    \]
    for every $x\in M$.
\end{definition}

Note that $[M]_\rho\cong M$ in $\Ban_R$ for all $\rho>0$ but the same is not true in $\Banc_R$ unless $\rho=1$.

\begin{definition}
A Banach $R$-module $M$ is said \emph{non-archimedean} if its norm satisfies the ultrametric inequality
\[
\norm{x_1+x_2}\leq \max_{i=1,2}\norm{x_i}.
\]
 When $R$ is non-archimedean, the full-subcategory of $\Ban_R$ consisting of the non-archimedean Banach modules is denoted by $\Ban_R^\na$. Analogously, $\Bannac_R$ is the full sub-category of non-archimedean Banach modules of $\Banc_R$.
\end{definition}

The categories of contracting morphisms are relevant because they have all limits and colimits (\cite[Proposition 3.21]{dag}). If $\{M_i:i\in I\}$ is a family of Banach $R$-modules, its coproduct in $\Banc_R$, called \emph{contracting coproduct}, is the Banach $R$-module

\begin{equation*}
\coprodc_{i\in I} M_i
\end{equation*}

whose elements are tuples $x=(x_i)_{i\in I}$ with $x_i\in M_i$ for all $i\in I$ satisfying

\begin{equation*}
\sum_{i\in I} \norm{x_i}_{M_i}<\infty.
\end{equation*}

The norm of $x$ is the sum

\begin{equation*}
\sum_{i\in I} \norm{x_i}_{M_i}
\end{equation*}

The product in $\Banc_R$, called \emph{contracting product}, is 
the Banach $R$-module

\begin{equation*}
\prodc_{i\in I} M_i
\end{equation*}

whose elements are tuples $x=(x_i)_{i\in I}$ with $x_i\in M_i$ for all $i\in I$ satisfying

\begin{equation*}
\sup_{i\in I} \norm{x_i}_{M_i}<\infty.
\end{equation*}

The norm of $x$ is the supremum
\begin{equation*}
\sup_{i\in I} \norm{x_i}_{M_i}
\end{equation*}

When $R$ is non-archimedean and we consider the category $\Bannac_R$ instead, the contracting product is computed in the same way. The contracting coproduct, denoted by
\[
\coprodnac_{i\in I} M_i,
\]
is slightly different: it consists of the tuples $x=(x_i)_{i\in I}$ with $x_i\in M_i$ for all $i\in I$ such that for all $\varepsilon>0$, the set of indexes $i\in I$ for which $\norm{x_i}_{M_i}>\varepsilon$ is finite. The norm of $x$ is the supremum

\begin{equation*}
\sup_{i\in I} \norm{x_i}_{M_i}
\end{equation*}

In particular we will use the following constructions which generalize the familiar spaces of summable sequences and bounded sequence to an arbitrary Banach ring $R$. 

\begin{definition}\label{def:ell_modules}
Let $X$ be a set with a function $\rho:X\to\Rp$. Define 
\begin{align*}
\ell^1(X,\rho)&:=\coprodc_{x\in X} [R]_{\rho(x)}
\\
\ell^1_\na(X,\rho)&:=\coprodnac_{x\in X} [R]_{\rho(x)}
\\
\ell^\infty(X,\rho)&:=\prodc_{x\in X} [R]_{\rho(x)}
\end{align*}
We use the notation
$\ell^1_*(X,\rho)$ to indicate both $\ell^1(X,\rho)$ and $\ell^1_\na(X,\rho)$.
Denote by $e_x$ the element of either $\ell^1_*(X,\rho)$ or $\ell^\infty(X,\rho)$ which is equal to $1$ in the component corresponding to $x\in X$ and zero in every other component.
When $X=\N$ we drop $X$ from the notation and just write $\ell^1_*(\rho)$ or $\ell^\infty(\rho)$. We call elements of $\ell^1_*(\rho)$ \emph{summable} sequences with respect to the \emph{weight} $\rho$ and elements of $\ell^\infty(\rho)$ \emph{bounded} sequences with respect to the \emph{weight} $\rho$. When $\rho:\N \to \Rp$ is constant and equal to $1$ we also drop $\rho$ and write simply $\ell^1_*$ or $\ell^\infty$.
\end{definition}

The category $\Ban_R$ is a closed monoidal symmetric category with respect to the \emph{completed projective} tensor product (\cite[Proposition 3.17]{dag}).
Given two Banach modules $M,N$, their completed (projective) tensor product, denoted by
\[
M \wot_R N,
\]
is the completion of the algebraic tensor product $M\otimes_R N$ with respect to the norm
\begin{equation*}
\norm{t}_{M\wot_R N}
:=
\inf
\left\{
\sum_{i=1}^n \norm{x_i}_M \norm{y_i}_N
:
t=\sum_{i=1}^n x_i\otimes y_i
\right\}
\end{equation*}

where $t$ is an element of $M\otimes_R N$ and the infimum is taken over all representations of $t$ as a finite sum of simple tensors of the form $x\otimes y$ for $x\in M$ and $y\in N$. We will drop the base ring from the notation of the tensor product and write simply $\wot$ when it is clear the underlying category considered.
The unit of the monoidal product $\wot$ is the Banach module $R$. The internal-hom of two Banach modules $M,N$ is the Banach $R$-module $\ihom_R(M,N)$ whose elements are morphisms $f:M\to N$ and the norm is the operator norm
\[
\norm{f}=\sup_{x\in M\setminus\{0\}} \frac{\norm{f(x)}_N}{\norm{x}_M}.
\]
The adjunction between the completed tensor product and the internal-hom is analogous to the algebraic adjunction: a morphism $g:M\wot N \to L$ of Banach modules corresponds to the morphism $f:M \to \ihom_R(N,L)$ satisfying
\[
f(x)(y)=g(x\otimes y)
\]
for all $x\in M$ and $y\in N$. The symmetric closed monoidal structure of $\Ban_R$ induces on $\Banc_R$ a symmetric closed monoidal structure using the same definitions. In particular, given a set $I$, a function $\rho:I\to\Rp$, a Banach module $N$ and a family of Banach modules $\{M_i:i\in I\}$, there are natural isomorphisms
\begin{equation}\label{eq:interaction_contr_(co)prod_and_monoidal_str}
\begin{aligned}
\ihom_R\Big(N,\prodc_{i\in I}[M_i]_{\rho(i)}\Big)
&\cong
\prodc_{i\in I}\left[\ihom_R(N,M_i)\right]_{\rho(i)}
\\
\ihom_R\Big(\coprodc_{i\in I}[M_i]_{\rho(i)},N\Big)
&\cong
\prodc_{i\in I}\left[\ihom_R(M_i,N)\right]_{\rho(i)^{-1}}
\\
\Big( \coprodc_{i\in I}[M_i]_{\rho(i)} \Big)\ \wot\ N
&\cong
\coprodc_{i\in I}\left[ M_i \wot N \right]_{\rho(i)}
\end{aligned}
\end{equation}
in $\Banc_R$ and $\Ban_R$. Note that even if $\ihom_R(-,-)$ is the same in $\Ban_R$ and $\Banc_R$, the underlying set of $\ihom_R(-,-)$ coincides with $\Hom_R(-,-)$ only in the category $\Ban_R$. Everything said about the monoidal structure remains true also in the non-archimedean case provided one modify the norm of the completed tensor product: if $M,N$ are non-archimedean Banach modules, 
\[
\norm{t}_{M\wot N}
:=
\inf
\left\{
\max_{i=1,\dots, n} \norm{x_i}_M \norm{y_i}_N
:
t=\sum_{i=1}^n x_i\otimes y_i
\right\}
\]
for every $t\in M\wot N$.

\begin{prop}
The category $\Ban_R$ is pre-abelian, namely it is a linear category with zero-object, finite limits and finite colimits. The inclusion functor $\Banc_R\to\Ban_R$ commutes with them. Finite coproducts are naturally equivalent to finite products. For a morphism $f:M \to N$ in $\Ban_R$ the following holds
\begin{enumerate}[\rm i)]
    \item $f$ is a monomorphism if and only if it is injective;
    \item $f$ is an epimorphism if and only if $f(M)$ is dense in $N$;
    \item $f$ is an isomorphism if and only if $f$ is bijective and there is a constant $C\in\Rp$ such that
    \[
    \norm{x}_M\leq C\norm{f(x)}_N
    \]
    for all $x\in M$.
\end{enumerate}
\end{prop}

\begin{proof}
See \cite[Proposition 3.12, Proposition 3.14]{dag}. 
\end{proof}

The category $\Ban_R$ shares many properties with the category of algebraic modules over a ring but it lacks arbitrary limits and colimits such as infinite products and infinite coproducts. This is addressed by considering the Ind-completion of $\Ban_R$, denoted by $\Ind(\Ban_R)$. The reader is referred to \cite[Section 8]{sga} for a detailed treatment of Ind-completions. For the reader's convenience we recall the definition of Ind-completion in the case of $\Ban_R$. 

Let $\widehat{\Ban_R}$ be the category of presheaves of sets on the (small) category $\Ban_R$, i.e. functors $\Ban_R^{\rm op} \to \Set$. The Yoneda embedding
\[
\Ban_R \longrightarrow \widehat{\Ban_R},
\quad M \longmapsto \Hom_R(-,M)
\]
identifies the category $\Ban_R$ with the full subcategory of $\widehat{\Ban_R}$ consisting of representable functors.

\begin{definition}
The category $\Ind(\Ban_R)$ is the full-subcategory of $\widehat{\Ban_R}$ consisting of objects isomorphic to filtered colimits of representable functors, called \emph{ind-Banach} modules. These are functors $M:\Ban_R^{\rm op}\to \Set$ for which there is a filtered category $I$ and a diagram of Banach modules
\[
I \to \Ban_R, \quad i\mapsto M_i
\]
such that 
\[
M\cong \varinjlim_{i\in I} \Hom_R(-,M_i).
\]
The colimit on the right is denoted by
\[
\indlim_{i\in I} M_i.
\]
The set of morphisms between two ind-Banach $R$-modules $M,N$ is denoted by $\Hom_R(M,N)$.
\end{definition}

Given two ind-Banach modules $M=\indlim_i M_i$ and $N=\indlim_j N_j$, morphisms form $M$ to $N$ are morphisms in the category $\widehat{\Ban}_R$. It follows that
\[\begin{split}
\Hom_R(M,N)
&=
\Hom_{\widehat{\Ban}_R}
\left(
\varinjlim_{i\in I}\Hom_R(-,M_i) , \varinjlim_{j\in J}\Hom_R(-,N_j)
\right)
\\
&=
\varprojlim_{i\in I}\Hom_{\widehat{\Ban}_R}
\left(
\Hom_R(-,M_i) , \varinjlim_{j\in J}\Hom_R(-,N_j)
\right)
\\
&=
\varprojlim_i \varinjlim_j \Hom_R(M_i,N_j).
\end{split}\]
by the Yoneda lemma and the properties of limits and colimits in $\widehat{\Ban}_R$. When we consider the ind-banach module $M$ as a functor $M:\Ban_R^{\mathrm{op}}\to\Set$, we have by definition 
\[
M(B)=\Hom_R(B,M)
\]
for every Banach module $B$. In general, every functor $M\in\widehat{\Ban}_R$ can be written as the colimit
\[
\varinjlim_{B\in\Ban_R/M} \Hom_R(-,B)
\]
where $\Ban_R/M$ is the (small) category consisting of natural transformations
\[
\Hom_R(-,B) \to M
\]
as objects, where $B\in\Ban_R$, and commutative triangles as morphisms. Then $M$ belongs to $\Ind(\Ban_R)$ if and only if $\Ban_R/M$ is filtered.
Suppose we have a morphism $f:M \to N$ between ind-Banach modules. Since $f$ is a natural transformation between functors, it is an isomorphism if and only if for every Banach module $B$ the induced map
\[
\Hom_R(B,f):\Hom_R(B,M) \to \Hom_R(B,N),\quad g \mapsto f\circ g 
\]
is bijective. Moreover, we can always construct a filtered category $K$, two diagrams
\begin{gather*}
K \to \Ban_R,\quad k \mapsto M_k,
\\
K \to \Ban_R,\quad k\mapsto N_k,
\end{gather*}
and a morphism of diagrams $(f_k)_{k\in K}$ such that 
\[
M=\indlim_{k\in K} M_k,
\qquad
N=\indlim_{k\in K} N_k
\]
and the morphism $f$ is equivalent to the inductive limit of $(f_k)_{k\in K}$ (see \cite[Remark 2.3]{dag}).

\begin{prop}
The category $\Ind(\Ban_R)$ is a pre-abelian with all limits and all colimits. The closed symmetric monoidal structure of $\Ban_R$ induces a closed symmetric monoidal structure on $\Ind(\Ban_R)$ such that the Yoneda embedding $\Ban_R \to \Ind(\Ban_R)$ is monoidal. Given 
\[
M=\indlim_{i\in I} M_i,
\quad
N=\indlim_{j\in J} N_j
\]
the tensor product and internal-hom are computed as follows
\begin{align*}
M\wot N &= \indlim_{(i,j)\in I\times J} M_i\wot N_j,
\\
\ihom_R(M,N) &= \varprojlim_{i\in I}\indlim_{j\in J} \ihom_R(M_i,N_j).
\end{align*}
\end{prop}

\begin{proof}
For limits and colimits see \cite[Proposition 8.9.1]{sga}. For the monoidal structure see \cite[Proposition 2.1.19]{quab}. To check that the latter result applies to the case of $\Ind(\Ban_R)$ see \cite[Proposition 3.15, Lemma 3.27]{dag}. 
\end{proof}

The category $\Ind(\Ban_R)$ shares many properties with the category of modules over a ring. It is suitable for homological algebra in the language of quasi-abelian categories \cite{quab}. While we will not deploy this formalism in its full generality, there are subtle differences from classical abelian categories to keep in mind. First, a morphism $f:M\to N$ in $\Ban_R$ or $\Ind(\Ban_R)$ can have trivial kernel and trivial cokernel even if it is not an isomorphism. Second, we need to be careful with the meaning of projective or flat objects.
We can always make the following commutative diagram out of $f$:
\begin{equation}\label{eq:(co)ker_(co)image_diagram}
\begin{tikzcd}
\Ker(f) \ar[r] 
&
M 
\ar[r,"f"] \ar[d]
& 
N 
\ar[r] \ar[d]
& 
\Coker(f) 
\\
 & \Coim(f) \ar[r, "\tilde{f}"] & \im(f) & 
\end{tikzcd}
\end{equation}
and $\tilde{f}$ is a monomorphism and an epimorphism. 

\begin{definition}
A morphism $f$ in $\Ind(\Ban_R)$ is said \emph{strict} if the induced morphism $\tilde{f}:\Coim(f) \to \im(f)$ fitting in the diagram \eqref{eq:(co)ker_(co)image_diagram} is an isomorphism. 
\end{definition}

\begin{definition}
An object $P\in \Ind(\Ban_R)$ is \emph{projective} if the functor
\[
\Hom_R(P,-): \Ind(\Ban_R) \to \Set
\]
sends strict epimorphisms to surjective maps.
\end{definition} 

\begin{definition}\label{def:flat_object}
An object $E\in\Ind(\Ban_R)$ is \emph{flat} if the functor
\[
E\wot - : \Ind(\Ban_R) \to \Ind(\Ban_R)
\]
commutes with all finite limits.
\end{definition}

Bearing the above definitions in mind, we set out below a few homological properties of $\Ban_R$ that we'll need in our computations.

\begin{lemma}\label{lem:ell1_is_flat_proj}
For every set $X$ with a map $\rho:X\to \Rp$, the Banach module
\[
\ell^1_*(X,\rho)
\]
introduced in Definition \ref{def:ell_modules} is projective and flat.
\end{lemma}

\begin{proof}
The fact that $\ell^1_*(X,\rho)$ is projective is proved in \cite[Lemma 3.38]{fremod}. In \cite[Proposition 3.13]{sheafy} it is proved that $\ell^1_\na(X,\rho)$ is flat in the sense of Definition \ref{def:flat_object} when $R$ is non-archimedean but the same proof works also for $\ell^1(X,\rho)$ and $R$ general.
\end{proof}

\begin{prop}\label{pr:proj_resol_of_Banach}
Every Banach module $M$ is the cokernel of a morphism
\[
P_1 \to P_0
\]
between flat projective Banach modules.
\end{prop}

\begin{proof}
Let $M\in\Ban_R$. Consider
\[
P_0:=\ell^1_*\big( M\setminus\{0\}\,,\,\norm{-}_M \big).
\]
There is a unique morphism $f:P \to M$ such that
\[
f(e_x)=x
\]
for all $x\in M\setminus\{0\}$. By \cite[Lemma 3.27]{dag}, the morphism $f$ is a strict epimorphism. Therefore $M$ is the cokernel of the kernel 
\[
\iota:\Ker(f) \to P_0.
\]
We can repeat the same construction for the Banach module $\Ker(f)$ obtaining a flat projective $P_1$ with a strict epimorphism
\[
g:P_1 \to \Ker(f).
\]
Then $\iota\circ g: P_1 \to P_0$ is a strict morphism whose cokernel is the same as the cokernel of $f$, which is $M$. By Lemma \ref{lem:ell1_is_flat_proj} we can conclude that $P_0,P_1$ are flat and projective.
\end{proof}

Besides the abstract definition, many objects of $\Ind(\Ban_R)$ admit a more concrete description in terms of $R$-modules with additional structure.

\begin{definition}
An object $M$ of $\Ind(\Ban_R)$ is called a (complete) \emph{bornological} $R$-module if it is isomorphic to the colimit of a \emph{monomorphic} filtered inductive system of Banach $R$-modules; i.e., there is a directed set $I$ and a diagram of Banach modules
\[
I \to \Ban_R,\quad i \mapsto M_i
\]
such that 
\[
M\cong\indlim_{i\in I} M_i
\]
and for every pair of indices $i,j\in I$ with $i<j$, the induced morphism
\[
M_i \to M_j
\]
of the diagram is an injective map. The full subcategory of $\Ind(\Ban_R)$ consisting of bornological $R$-modules is denoted by $\Born_R$.
\end{definition}

\begin{prop}\label{pr:born_is_concrete_reflexive_subcat}
The category $\Born_R$ has all limits and colimits. It is a reflective subcategory of $\Ind(\Ban_R)$ in the sense that the embedding functor $\Born_R \to \Ind(\Ban_R)$ has a left adjoint. 
The forgetful functor $\Ban_R\to \Set$ extends to a functor 
\[
U:\Ind(\Ban_R) \to \Set
\]
which commutes with filtered colimits and is the essentially unique extension of the forgetful functor with this property. The restriction of $U$ to the full subcategory $\Born_R$ is faithful and it has a left adjoint
\[
\Set \to \Born_R , 
\quad 
X \mapsto \coprod_{x\in X} R.
\]
\end{prop}

\begin{proof}
The left adjoint of the embedding $\iota:\Born_R \to \Ind(\Ban_R)$ is constructed in the Appendix of \cite[Proposition A.11]{condborn}. For an object 
\[
M=\indlim_{i\in I} M_i
\]
in $\Ind(\Ban_R)$, the left adjoint of $\iota$ sends $M$ to the bornological $R$-module
\[
M^{\mathrm{sep}}:=\indlim_{i\in I} \Coim(\alpha_i)
\]
where $\alpha_i:M_i \to M$ is the morphism obtained by the universal property of the colimit. The existence of all limits and colimits in $\Born_R$ follows from the fact that it is a reflective subcategory of a category with all limits and colimits. 
If $F:I \to \Born_R$ is a diagram, every cone $C$ of $\iota\circ F$ factors through the unit of the adjunction $C\to \iota(C^{\mathrm{sep}})$. Therefore the limit cone of $\iota\circ F$ must be bornological. For the existence of colimits, one can check that 
\[
\left( \colim \iota\circ F \right)^{\mathrm{sep}}
\]
satisfies the universal property of colimits in $\Born_R$.
The forgetful functor of $\Ban_R$ extends to $\Ind(\Ban_R)$ by the universal property of ind-categories and the fact that $\Set$ has all filtered colimits (\cite[Proposition 8.7.3]{sga}). The proof that the restriction of $U$ to $\Born_R$ is faithful can be found in \cite[Proposition 3.31]{dag}. Given that $U:\Ind(\Ban_R)\to \Set$ is represented by $R$ and that $\Ind(\Ban_R)$ has all coproducts, it is easy to see that the free module-construction 
\[
\mathrm{Free}:\Set \to \Ind(\Ban_R),\quad X \to \coprod_{x\in X} R
\]
is the left adjoint of $U$. We can express the free module over $X\in\Set$ as follows
\[
\coprod_{x\in X} R
\cong
\indlim_{F\subseteq X,\text{ $F$ finite}} \coprod_{x\in F} R
\]
where the coproduct $\coprod_{x\in F}$ is computed in $\Ban_R$. From this expression we see that $\coprod_{x\in X} R$ is bornological. Therefore $\mathrm{Free}$ factor through $\iota$ and it is left adjoint to $U\circ\iota$.
\end{proof}

We now explain the meaning of Proposition \ref{pr:born_is_concrete_reflexive_subcat} and the advantage of considering bornological modules. Recall that monomorphisms in $\Ban_R$ coincide with injective bounded maps. Given a bornological module $M$ as a monomorphic inductive limit
\[
M=\indlim_{i\in I} M_i
\]
we have 
\[
U(M)=\varinjlim_{i\in I} U(M_i)
\]
and the maps of the inductive system $U(M_i)\to U(M_j)$ for $i<j$ are injective. The set $U(M)$ is essentially the increasing union of the underlying sets of the Banach modules $M_i$ as $i\in I$ increase. This means that a bornological module is uniquely identified by taking an (algebraic) $R$-module $M$ and specifying a family $\{M_i:i\in I\}$ of (algebraic) sub-modules of $M$ and a family of norms $\norm{-}_{M_i}:M_i \to \R_+$ such that
\begin{enumerate}[i)]
    \item $I$ is a directed set;
    \item $(M_i,\norm{-}_{M_i})$ is a Banach $R$-module;
    \item for every $i<j$ in $I$, the sub-module $M_i$ is included in $M_j$ and the map $M_i \to M_j$ induced by the inclusion is bounded: there is a constant $C\in\R_+$ such that for all $x\in M_i$
    \[
    \norm{x}_{M_j}\leq C\norm{x}_{M_i}.
    \]
\end{enumerate}
Also morphisms between bornological modules have a concrete description in terms of linear maps preserving a structure. Suppose $M,N$ are bornological modules given as unions of Banach modules 
\[
M=\bigcup_{i\in I}M_i,\quad N=\bigcup_{j\in J} N_j.
\]
Then a morphism $f:M \to N$ in $\Born_R$ is the same as an $R$-linear map $f:M \to N$ satisfying:
\begin{enumerate}[i)]
    \item for every $i\in I$ there is a $j\in J$ such that $f(M_i)\subseteq N_j$;
    \item the restriction $M_i \to N_j,\ x\mapsto f(x)$ is bounded: there is a constant $C(i,j)\in\R_+$ depending on $i$ and $j$ such that for all $x\in M_i$
    \[
    \norm{f(x)}_{N_j}\leq C(i,j)\norm{x}_{M_i}.
    \]
\end{enumerate}
The fact that the forgetful functor $U:\Born_R \to \Set$ has a left adjoint implies that it commutes with all limits. Thanks to this we know that the underlying $R$-module of a limit of bornological modules is just the limit of the corresponding underlying $R$-modules. For example, the underlying $R$-module of the product 
\[
\prod_{i\in I} M_i
\]
of a family of bornological modules $\{M_i: i\in I\}$ is just the product
\[
\prod_{i\in I} U(M_i).
\]
Another important example is the kernel of a morphism $f:M \to N$ in $\Born_R$. Its underlying $R$-module is just 
\[
\left\{ x\in M: f(x)=0 \right\}.
\]
In particular, $f$ is a monomorphism if and only if it is injective as a map of sets.

\subsection{The duality functor and Nuclear modules}
Consider the functor 
\[
\ihom_R(-,R):\Ind(\Ban_R)^{\mathrm{op}} \to \Ind(\Ban_R).
\]
For an $M\in\Ind(\Ban_R)$, we adopt the notation
\[
M^\vee:=\ihom_R(M,R).
\] 
For every $M\in\Ind(\Ban_R)$ there is a canonical morphism
\begin{equation*}
\mathrm{ev}_M:M^\vee \wot M \to R,
\quad
\xi\otimes x \mapsto \langle \xi,x\rangle:=\xi(x)
\end{equation*}
corresponding to the identity of $M^\vee$ under the adjunction
\[
\Hom_R(M^\vee\wot M, R)\cong\Hom_R(M^\vee,\ihom_R(M,R))=\End_R(M^\vee).
\]
Since the monoidal product $\wot$ is symmetric, we obtain by the same adjunction also a natural morphism
\[
M \to M^{\vee\vee},
\quad
x\mapsto \langle-,x\rangle
\]

\begin{definition}
Let $M\in\Ind(\Ban_R)$. The object $M^\vee$ is called the \emph{dual} of $M$. The ind-Banach module $M$ is said \emph{reflexive} if the natural morphism
\[
M \to M^{\vee\vee}
\]
is an isomorphism. Given two objects $M,N\in\Ind(\Ban_R)$ and a morphism 
\[
g:M\wot N \to R
\]
we say that they are \emph{dual via }$g$ if the two morphisms 
\[
M \to N^\vee, \quad N \to M^\vee
\]
induced by $g$ through adjunction are isomorphisms.
\end{definition}

Given two objects $M,N\in\Ind(\Ban_R)$, there is a canonical morphism 
\begin{equation}\label{eq:canonical_morphism_nuclearity}
M^\vee \wot N \to \ihom_R(M,N)   
\end{equation}
corresponding to the morphism
\[
M^\vee\wot N \wot M \to N, 
\quad
\xi\otimes y \otimes x \mapsto \langle \xi,x\rangle y
\]
under the tensor-hom adjunction. The latter is obtained composing $\mathrm{ev}_M\wot\id_N$ with the isomorphism
\[
M^\vee\wot M \wot N \cong M^\vee \wot N \wot M.
\]

\begin{definition}
A morphism $f:M\to N$ between Banach $R$-modules is said \emph{nuclear} if it belongs to the set-theoretic image of the morphism \eqref{eq:canonical_morphism_nuclearity}, or equivalently, if it belongs to the image of the map
\[
\Hom_R(R,M^\vee \wot N) \to \Hom_R(M,N).
\]
Concretely, $f$ is nuclear if there is a sequence $(\xi_n)_{n\in\N}$ of elements of $M^\vee$ and a sequence $(y_n)_{n\in\N}$ of elements of $N$ such that 
\begin{gather*}
\sum_{n=0}^\infty \norm{\xi_n}_{M^\vee}\norm{y_n}_N <\infty,
\\
\left(\;
\lim_{n\to\infty}\ \norm{\xi_n}_{M^\vee}\norm{y_n}_N = 0
\quad
\text{in the non-archimedean case}
\; \right)
\end{gather*}
and for every $x\in M$
\[
f(x)=\sum_{n=0}^\infty \langle \xi_n,x\rangle y_n.
\]
We say that the tensor 
\[
\sum_{n=0}^\infty \xi_n \otimes y_n \in M^\vee\wot N
\]
\emph{represents} the morphism $f$.
\end{definition}

In the next Proposition are collected the stability properties of nuclear morphisms

\begin{prop}
In the category $\Ban_R$, the following holds:
\begin{enumerate}[\rm i)]
    \item if $f:M\to N$ is a nuclear morphism, then 
    \[
    f^\vee:N^\vee\to M^\vee
    \]
    is also nuclear;
    \item if $f:M\to N$ is nuclear and $g:M'\to M,\;h:N \to N'$ are any morphisms in $\Ban_R$, then $f\circ g$ and $h\circ f$ are nuclear;
    \item if $f_1:M_1 \to N_1$ and $f_2:M_2 \to N_2$ are nuclear morphisms then 
    \[
    f_1\wot f_2 : M_1 \wot M_2 \to N_1 \wot N_2
    \]
    is nuclear.
\end{enumerate}
\end{prop}

\begin{proof}
The above properties are proved in \cite[Proposition 2.2 and 2.3]{nucle} in greater generality. 
\end{proof}

Now we investigate nuclearity for morphisms between modules of the form $\ell^1_*(\rho)$ where $\rho:\N \to \Rp$ is a sequence of strictly positive real numbers. Suppose $\sigma:\N \to \Rp$ is another sequence. The set of morphisms 
\[
\ell^1_*(\rho) \to \ell^1_*(\sigma)
\]
is the underlying set of the Banach module $\ihom_R\big(\ell^1_*(\rho),\ell^1_*(\sigma)\big)$. We can compute the internal-hom in $\Banc_R$ using the universal properties of contracting coproducts in $\Banc_R$ (and $\Bannac_R$):
\begin{equation}\label{eq:matrix_ell1}
\begin{split}
\ihom_R\big(  \ell^1_*(\rho), \ell^1_*(\sigma)  \big)
&\cong
\prodc_{i\in\N}\ihom_R\big(  [Re_i]_{\rho(i)}, \ell^1_*(\sigma)  \big)
\\
&\cong
\prodc_{i\in \N} \Big( 
[R\,\check{e}_i]_{\rho(i)^{-1}}\wot \coprodc_{j\in\N} [R\,e_j]_{\sigma(j)}
\Big)
\\
&\cong
\prodc_{i\in\N} \coprodc_{j\in\N} 
[R\,\check{e}_i\otimes e_j]_{\frac{\sigma(j)}{\rho(i)}}
\end{split}
\end{equation}
Here $\check{e}_i$ denotes the projection $\ell^1_*(-) \to R$ to the $i$-th component and $\check{e}_i\otimes e_j$ corresponds to the unique morphism $R\,e_i \to R\,e_j$ sending $e_i$ to $e_j$. In \eqref{eq:matrix_ell1} we are identifying a morphism ${f:\ell^1_*(\rho) \to \ell^1_*(\sigma)}$ with a matrix
\[
\big(a_{ij} \big)_{i,j\in\N},\quad a_{ij}:=\langle \check{e}_j, f(e_i)\rangle,
\]
satisfying the condition
\begin{equation*}
\sup_{i\in\N} \sum_{j\in\N} \abs{a_{ij}}\frac{\sigma(j)}{\rho(i)} <\infty
\end{equation*}
or, in the non-archimedean case, the conditions
\begin{align*}
    &\lim_{j\to\infty} \abs{a_{ij}}\sigma(j)=0 
    \quad
    \text{for all $i\in\N$, and} \\
    &\sup_{i\in\N}\,\sup_{j\in\N}\,\abs{a_{ij}} \frac{\sigma(j)}{\rho(i)} <\infty.
\end{align*}
For the tensor product we have
\begin{equation}\label{eq:nuclear_matrix_ell1}
\begin{split}
\ell^1_*(\rho)^\vee \wot\, \ell^1_*(\sigma) 
&\cong
\coprodc_{j\in\N} \ell^1_*(\rho)^\vee \wot [R\,e_j]_{\sigma(j)}
\\
&\cong
\coprodc_{j\in\N} \Big( 
\prodc_{i\in\N} [R\,\check{e}_i]_{\rho(i)^{-1}}
\Big)
\wot [R\,e_j]_{\sigma(j)}
\\
&\cong
\coprodc_{j\in\N} \prodc_{i\in\N} 
[R\,\check{e}_i\otimes e_j]_{\frac{\sigma(j)}{\rho(i)}}.
\end{split}
\end{equation}
With \eqref{eq:matrix_ell1} and \eqref{eq:nuclear_matrix_ell1} in mind we observe that nuclear morphisms correspond to those matrices
\[
\big(a_{ij} \big)_{i,j\in\N},\quad a_{ij}\in R
\]
such that 
\[
\sup_{i\in\N} \frac{\abs{a_{ij}}}{\rho(i)}<\infty
\]
for all $j\in\N$ and
\begin{align*}
&\sum_{j=0}^\infty \sup_{i\in\N} \abs{a_{ij}}\frac{\sigma(j)}{\rho(i)} <\infty,
\quad 
\text{or}
\\
& \lim_{j\to\infty} \sup_{i\in\N} \abs{a_{ij}}\frac{\sigma(j)}{\rho(i)} = 0
\quad
\text{in the non-archimedean case.}
\end{align*}

We are mainly interested in the case of the identity matrix, in other words we seek the conditions that $\rho,\sigma$ must satisfy in order to have a well-defined nuclear morphism
\[
\pi_{\sigma\rho}:\ell^1_*(\rho) \to \ell^1_*(\sigma)
\]
such that $\pi_{\sigma\rho}(e_n)=e_n$ for all $n\in\N$.

\begin{prop}\label{pr:nuclear_morph_ell}
Let $\rho,\sigma:\N \to \Rp$ be two sequences of strictly positive real numbers. Suppose that 
\[
\sup_{n\in\N}\frac{\sigma(n)}{\rho(n)}<\infty,
\]
then 
\[
\pi_{\sigma\rho}: \ell^1_*(\rho) \to \ell^1_*(\sigma),
\quad e_n \mapsto e_n
\]
is well-defined and it is nuclear if and only if
\begin{gather*}
\sum_{n=0}^\infty \frac{\sigma(n)}{\rho(n)}<\infty
\\
\Big(\, \lim_{n\to\infty} \frac{\sigma(n)}{\rho(n)} = 0
\quad \textrm{in the non-archimedean case.}\, \Big)
\end{gather*}
Moreover, when $\pi_{\sigma\rho}$ is nuclear, it extends to a well-defined nuclear morphism
\begin{equation*}
    \tau_{\sigma\rho}: \ell^\infty(\rho) \to \ell^1_*(\sigma).
\end{equation*}
\end{prop}

\begin{proof}
Consider the matrix
\[
\delta:=\big( \delta_{ij} \big)_{i,j\in\N},
\]
where
\[
\delta_{ij}=\begin{cases}
1 & \text{if $i=j$;} \\
0 & \text{otherwise.}
\end{cases}
\]
Under the isomorphism in \eqref{eq:matrix_ell1}, $\delta$ induces a well-defined morphism $\pi_{\sigma\rho}:\ell^1(\rho)\to\ell^1(\sigma)$ by the condition
\[
\sup_{i\in\N}\sum_{j\in\N}\abs{\delta_{ij}}\frac{\sigma(j)}{\rho(i)}
=\sup_{i\in\N} \frac{\sigma(i)}{\rho(i)}<\infty.
\]
The non-archimedean case is analogous. Moreover, from \eqref{eq:nuclear_matrix_ell1} we see that $\pi_{\sigma\rho}$ is nuclear if and only if its matrix $\delta$ satisfies
\[
\sum_{j=0}^\infty \sup_{i\in\N}\abs{\delta_{ij}} \frac{\sigma(j)}{\rho(i)}<\infty.
\]
but the above quantity is precisely
\[
\sum_{j=0}^\infty \frac{\sigma(j)}{\rho(j)}.
\]
The non-archimedean case is analogous. 

Suppose $\pi_{\sigma\rho}$ is nuclear, then it is represented by the tensor
\[
\sum_{n=0}^\infty \check{e}_n\otimes e_n
\]
under the natural morphism
\[
\ell^1_*(\rho)^\vee\wot\ell^1_*(\sigma)
\to
\ihom_R\big(\, \ell^1_*(\rho),\,\ell^1_*(\sigma) \,\big).
\]
There is a canonical morphism from the coproduct to the product in every category with zero object. In particular, there is a canonical contracting morphism
\[
c_\rho: \ell^1_*(\rho) \to \ell^\infty(\rho).
\]
It is the identity on elements and expresses the fact that summable sequences with respect to the weight $\rho$ are bounded with respect the same weight. Applying the natural morphism 
\[
(-)^\vee \wot \ell^1_*(\sigma) \to \ihom_R\big(-,\ell^1_*(\sigma)\big)
\]
to $c_\rho$ yields the commutative square
\begin{equation}\label{eq:nucl_comm_square_for_c_rho}
\begin{tikzcd}
\ihom_R\big( \ell^\infty(\rho),\,\ell^1_*(\sigma) \big)
\ar[r,"-\circ c_\rho"]
&
\ihom_R\big( \ell^1_*(\rho),\,\ell^1_*(\sigma) \big)
\\
\ell^\infty(\rho)^\vee \wot \ell^1_*(\sigma)
\ar[u] \ar[r,"c_\rho^\vee\wot\id_{\ell^1_*(\sigma)}"]
&
\ell^1_*(\rho)^\vee \wot \ell^1_*(\sigma)
\ar[u]
\end{tikzcd}
\end{equation}
If we denote the projection $\ell^\infty(-) \to R$ to the $n$-th component by $\check{e}'_n$, then 
\[
\check{e}_n=c_\rho^\vee(\check{e}'_n).
\]
We have also that 
\[
\norm{\check{e}_n}_{\ell^1(\rho)^\vee}
=\rho(n)^{-1}=
\norm{\check{e}'_n}_{\ell^\infty(\rho)^\vee}.
\]
Therefore, the tensor
\[
\sum_{n=0} \check{e}'_n\otimes e_n
\]
is a well-defined element of $\ell^\infty(\rho)^\vee \wot \ell^1_*(\sigma)$, and 
\[
c_\rho^\vee\wot\id_{\ell^1_*(\sigma)}\Big(
\sum_{n=0} \check{e}'_n\otimes e_n
\Big)
=
\sum_{n=0}^\infty \check{e}_n\otimes e_n.
\]
If $\tau_{\sigma\rho}$ denotes the nuclear morphism corresponding to $\sum_{n=0}^\infty \check{e}'_n\otimes e_n$, we must have 
\[
\tau_{\sigma\rho}\circ c_\rho=\pi_{\sigma\rho}
\]
from the commutativity of the square \eqref{eq:nucl_comm_square_for_c_rho}.
\end{proof}

The morphism $\tau_{\sigma\rho}:\ell^\infty(\rho)\to\ell^1_*(\sigma)$ is one of the key to the reflexivity results of the next section because it allows us to exchange weighted $\ell^1_*$ modules with weighted $\ell^\infty$ modules inside certain categorical limits and colimits.

Now we introduce the nuclearity property for ind-Banach modules as in \cite{fremod}.

\begin{prop}\label{pr:equivalent_defs_of_nuclear_ind_ban}
Let $M$ be an object of $\Ind(\Ban_R)$. The following are equivalent:
\begin{enumerate}[\rm i)]
    \item the natural morphism 
    \[
     B^\vee \wot M \to \ihom_R(B,M)
    \]
    is an isomorphism for every Banach $R$-module $B$;
    \item every filtered inductive system 
    \[
    I \longrightarrow \Ban_R,
    \quad (i\xrightarrow{\alpha} j)
    \longmapsto \big(M_i\xrightarrow{M(\alpha)} M_j\big)
    \]
    representing $M$ has the property that for every $i\in I$ there is a morphism $\alpha:i\to j$ in $I$ such that 
    \[
    M(\alpha):M_i \to M_j
    \]
    is nuclear;
    \item 
    there exists a filtered inductive system 
    \[
    I \longrightarrow \Ban_R,
    \quad (i\xrightarrow{\alpha} j)
    \longmapsto \big(M_i\xrightarrow{M(\alpha)} M_j\big)
    \]
    such that 
    \[
    M\cong\indlim_{i\in I} M_i
    \]
    and for every $i\in I$ there is a morphism $\alpha:i\to j$ in $I$ such that 
    \[
    M(\alpha):M_i \to M_j
    \]
    is nuclear.
\end{enumerate}
\end{prop}

\begin{proof}
If i) is true then ii) follows by taking $B=M_i$ and by observing that the surjectivity of the morphism 
\[B^\vee \wot M\to\ihom_R(B,M)\] 
is equivalent to the fact that for every morphism
\[
f:B \to M_i
\]
there is an $\alpha:i\to j$ such that 
\[
M(\alpha)\circ f : B \to M_j
\]
is representable by a tensor in $B^\vee\wot M_j$. See \cite[Lemma 4.14 and 4.16]{fremod} for the details. 
The implication $\mathrm{ii})\Rightarrow \mathrm{iii})$ is obvious and iii) implies i) by \cite[Lemma 4.15]{fremod}.
\end{proof}

\begin{definition}
An object $M\in\Ind(\Ban_R)$ is said \emph{nuclear} if it satisfies one of the three equivalent conditions of Proposition \ref{pr:equivalent_defs_of_nuclear_ind_ban}. A filtered inductive system $I\to\Ban_R$ like in ii) or iii) of Proposition \ref{pr:equivalent_defs_of_nuclear_ind_ban} is said to have \emph{nuclear transition morphisms}.
\end{definition}

\begin{prop}\label{pr:nuclears_are_flat}
Let $M\in\Ind(\Ban_R)$. Then the following are equivalent: 
\begin{enumerate}[\rm i)]
\item $M$ is nuclear;
\item there is a filtered category $I$, a family $\{\rho_i:i\in I\}$ of sequences $\rho_i:\N\to\Rp$ and an inductive system
\[
I \longrightarrow \Ban_R,
\quad 
(i\xrightarrow{\alpha} j)
\longmapsto 
\big(
\ell^1_*(\rho_i)
\xrightarrow{M(\alpha)} 
\ell^1_*(\rho_j)
\big)
\]
with nuclear transition morphisms such that 
\[
M=\indlim_{i\in I} \ell^1_*(\rho_i);
\]
\item there is a filtered category $I$, a family $\{\rho_i:i\in I\}$ of sequences $\rho_i:\N\to\Rp$ and an inductive system
\[
I \longrightarrow \Ban_R,
\quad 
(i\xrightarrow{\alpha} j)
\longmapsto 
\big(
\ell^\infty(\rho_i)
\xrightarrow{M(\alpha)} 
\ell^\infty(\rho_j)
\big)
\]
with nuclear transition morphisms such that 
\[
M=\indlim_{i\in I} \ell^\infty(\rho_i).
\]
\end{enumerate}
In particular, any nuclear ind-Banach module is flat.
\end{prop}

\begin{proof}
Given ii) or iii) we have i) by Proposition \ref{pr:equivalent_defs_of_nuclear_ind_ban}. The implication i)$\Rightarrow$ii) is proved in \cite[Lemma 4.19]{fremod}. The implication i)$\Rightarrow$iii) is proved in \cite[Lemma 4.20]{fremod}. A nuclear ind-Banach module $M$ is isomorphic to a filtered colimit of flat Banach modules by ii). Then $M$ is flat because filtered colimits commute with finite limits in $\Ind(\Ban_R)$ (see \cite[Proposition 8.9.1]{sga}). 
\end{proof}

\subsection{Fréchet modules}
\begin{definition}
    A \emph{Fréchet} $R$-module is a bornological $R$-module $F$ which is the limit of a sequential projective system of Banach $R$-modules, that is 
    \[
    F=\varprojlim \;
    (\, \cdots\xrightarrow{\pi_{2,3}} F_2 \xrightarrow{\pi_{1,2}} F_1\xrightarrow{\pi_{0,1}} F_0 \,)
    \]
    with $\pi_{j,j+1}:F_{j+1}\to F_j$ morphism of Banach $R$-modules for every $j\in\N$. 
    The projection morphisms from the limit to the components are denoted by
    \[
    \pi_j: F \to F_j.
    \]
    for every $j\in\N.$
\end{definition}
\begin{remark}
A Fréchet module is equivalently a limit in $\Ind(\Ban_R)$ of a sequential projective system consisting of morphisms in $\Banc_R$. To see that, suppose 
\[
F=\varprojlim_{j\in\N} F_j
\]
for a projective system $\pi_{ij}:F_j\to F_i$. Define the new norm on $F_j$
\[
\norm{x}_j':=\max_{i=0,\dots, j} \norm{\pi_{ij}(x)}_{F_i}
\quad 
\text{for every }x\in F_j.
\]
It's clear that $\norm{x}_{F_j}\leq\norm{x}_j'$, but it also holds the other inequality
\[
\norm{x}_j'\leq C \norm{x}_{F_j}
\quad
\text{for all }x\in F_j,
\]
where $C$ is the maximum among the operator-norms $\norm{\pi_{ij}}$ for $i=0,\dots,j$.
Thus the new norm $\norm{-}_j'$ is equivalent to the original norm of $F_j$ for all $j\in\N$. Moreover, for all $i,j\in\N$ satisfying $i\leq j$ and all $x\in F_j$
\[\begin{split}
\norm{\pi_{ij}(x)}_i' 
&= \max_{k=0,\dots, i} \norm{\pi_{ki}(\pi_{ij}(x))}_{F_k} \\
&= \max_{k=0,\dots, i} \norm{\pi_{kj}(x)}_{F_k} \\
&\leq \max_{k=0,\dots, j} \norm{\pi_{kj}(x)}_{F_k} \\
&= \norm{x}_j'.
\end{split}\]
Hence, $\pi_{ij}$ becomes a morphism of $\Banc_R$ if $F_i$ has the norm $\norm{-}_i'$ and $F_j$ has the norm $\norm{-}_j'$.
\end{remark}

In view of the previous remark we will implicitly assume that a Fréchet module has contractive morphisms in the projective system. Note that by Proposition \ref{pr:born_is_concrete_reflexive_subcat} the embedding of $\Born_R$ in $\Ind(\Ban_R)$ commutes with limits. Thus, the notion of Fréchet modules remains the same whether they are considered in $\Born_R$ or in $\Ind(\Ban_R)$.

\subsection{Metrizable modules}
Now we recall the notion of metrizable modules from \cite[Section 5]{fremod}. For a bornological module, it formalizes the condition that every \emph{countable} family of Banach sub-modules is contained in some Banach sub-module. For our purposes, metrizable modules are useful because they behave nicely with tensor products of countable limits in $\Ind(\Ban_R)$. 

\begin{definition}
A category $I$ is said $\aleph_1$-filtered if any countable subset of objects of $I$ admits a cocone in $I$.
\end{definition}

When $I$ is just a partially ordered set, $I$ being $\aleph_1$-filtered means that every countable subset of $I$ has an upper bound in $I$. Just as filtered colimits of sets commutes with finite limits, $\aleph_1$-filtered colimits of sets commutes with countable limits. 

\begin{definition}
An object $M\in\Ind(\Ban_R)$ is said \emph{metrizable} if the slice category $\Ban_R/M$ is $\aleph_1$-filtered.
\end{definition}

\begin{prop}\label{pr:lim-wot-compatibility_for_metrizables}
Let $N\in\Ind(\Ban_R)$ be metrizable. If $I$ is a countable set and $\{M_i:i\in I\}$ is a countable family of objects in $\Ind(\Ban_R)$, then the natural morphism
\[
\Big( \prod_{i\in I} M_i \Big) \wot N
\to 
\prod_{i\in I}\Big( M_i\wot N \Big)
\]
is an isomorphism. Suppose $I$ is a countable category and
\[
I \to \Ind(\Ban_R),\quad i\mapsto M_i
\]
is a diagram of ind-Banach modules. If $N$ is also flat, then the natural morphism
\[
\big( \varprojlim_{i\in I} M_i \big) \wot N
\to 
\varprojlim_{i\in I} \big( M_i \wot N \big)
\]
is an isomorphism.
\end{prop}
\begin{proof}
This is \cite[Lemma 5.18]{fremod}.
\end{proof}

In the next proposition, we collect the main examples of metrizable modules that we will encounter.

\begin{prop}\label{pr:some_metrizable_mods}
Let $M\in\Ind(\Ban_R)$ be a Fréchet $R$-module. Then $M$ is metrizable. In particular, products of at most countable families of Banach $R$-modules are metrizable.
\end{prop}

\begin{proof}
The prof can be found at \cite[Corollary 5.13]{fremod} for Fréchet modules. The other are special cases of Fréchet modules 
\end{proof}

%% file: sections/reflex.tex
In the following section we provide a family of bornological modules and we prove reflexivity for them. This family is big enough to include relative version of many classical spaces important in analytic geometry and functional analysis. For our purposes, this family includes spaces of power series over $R$ that are convergent on open disks or overconvergent on closed disks.
The following definition is inspired by the theory of K\"othe sequence spaces (so-called echelon and co-echelon spaces) over $\R$ or $\C$ (cf. \cite[\S 30, 8.]{koethe}).
\begin{definition}\label{def:sequential_modules}
Let $S$ be a directed set of functions $\rho:\N \to \Rp$ ordered by the relation
\begin{equation*}
\rho_1<\rho_2
\quad
\text{if and only if there is $n_0\in\N$ such that}
\quad
\rho_1(n)<\rho_2(n) 
\quad
\text{for all }n\geq n_0.
\end{equation*}
Define the \emph{sequential modules} associated to $S$ as
\begin{align*}
\lambda(S)&:=\varprojlim_{\rho\in S} \ell^1_*(\rho)
=\varprojlim_{\rho\in S} \coprodc_{n\in\N}[Re_n]_{\rho(n)},
\\
\kappa(S)&:=\varinjlim_{\rho\in S} \ell^1_*(\rho^{-1})
=\varinjlim_{\rho\in S} \coprodc_{n\in\N}[Re_n]_{\rho(n)^{-1}}.
\end{align*}
The structure morphisms of the projective and inductive systems are the inclusions
\begin{align*}
\pi_{\rho_1\rho_2} : 
\ell^1_*(\rho_2) \to \ell^1_*(\rho_1),
\quad &
e_n\mapsto e_n 
\\
\iota_{\rho_1\rho_2} :
\ell^1_*(\rho_1^{-1}) \to \ell^1_*(\rho_2^{-1}),
\quad &
e_n\mapsto e_n.
\end{align*}
respectively, for every $\rho_1<\rho_2$ in $S$.
\end{definition}

The module $\lambda(S)$ as a set consists of the sequences $(a_0,a_1,a_2,\dots)\in R^\N$ satisfying the condition
\begin{align*}
&\sum_{n=0}^\infty \abs{a_n}\rho(n) < \infty
\quad
\text{for all $\rho\in S$, or}
\\
&\lim_{n\to\infty} \abs{a_n}\rho(n) = 0
\quad
\text{for all $\rho\in S$ in the non-archimedean case}
\end{align*}
The underlying set of $\kappa(S)$ instead consists of the sequences $(a_0,a_1,a_2,\dots)\in R^\N$ for which there exists a $\rho\in S$ such that
\begin{align*}
&\sum_{n=0}^\infty \frac{\abs{a_n}}{\rho(n)} < \infty,
\quad \text{or}
\\
&\lim_{n\to\infty} \frac{\abs{a_n}}{\rho(n)} = 0 
\quad \text{in the non-archimedean case.}
\end{align*}

Now we will investigate reflexivity for these families of sequential modules. Take for example $\kappa(S)$. Its dual is 
\begin{equation}\label{eq:dual_of_k(S)}
\kappa(S)^\vee
\cong
\varprojlim_{\rho\in S}\left[\ell^1_*(\rho^{-1})\right]^\vee
\cong
\varprojlim_{\rho\in S}\ell^\infty(\rho)
\end{equation}
where we have used the isomorphism
\begin{equation*}
\prodc_{n\in\N}[R\check{e}_n]_{\rho(n)}
\cong
\Big( \coprodc_{n\in\N} [Re_n]_{\rho(n)^{-1}} \Big)^\vee
\end{equation*}
obtained from Equation \eqref{eq:interaction_contr_(co)prod_and_monoidal_str} in the particular case $I=\N,M_i=Re_i$ and $N=R$.
Note that the projective limit in the equation \eqref{eq:dual_of_k(S)} would be $\lambda(S)$ if it were possible to substitute $\ell^\infty(\rho)$ with $\ell^1_*(\rho)$. The following Lemma gives a sufficient condition to exchange $\ell^1_*$ and $\ell^\infty$ in the definition of $\lambda(S)$ and $\kappa(S)$.

\begin{lemma}\label{lem:prodc_coprodc_exchange_in_lambda-kappa}
Let $S$ be a directed set of sequences as in Definition \ref{def:sequential_modules}. Assume that there is an increasing map
\[
S \to S,
\quad
\rho \mapsto \rho_+
\]
satisfying 
\begin{gather*}
\rho(n)\leq \rho_+(n) \quad \text{for all $n\in\N$, and}
\\
\sum_{n=0}^\infty \frac{\rho(n)}{\rho_+(n)} < \infty
\\
\Big(\ \lim_{n\to\infty} \frac{\rho(n)}{\rho_+(n)} = 0 
\quad
\text{in the non-archimedean case}\ \Big).
\end{gather*}
Then, the canonical morphisms
\begin{align*}
& \lambda(S) \to \varprojlim_{\rho\in S} \ell^\infty(\rho),
\\
& \kappa(S) \to \varinjlim_{\rho\in S} \ell^\infty(\rho^{-1})
\end{align*}
are isomorphisms.
\end{lemma}
\begin{proof}
First note that, by Proposition \ref{pr:nuclear_morph_ell}, we have well-defined morphisms
\begin{align*}
&\tau_\rho:\ell^\infty(\rho_+) \to \ell^1_*(\rho)
\\
&\tau_{\rho^{-1}}:\ell^\infty(\rho^{-1}) \to \ell^1_*(\rho_+^{-1})
\end{align*}
which are the identity on elements. We prove the statement for the projective limit, the case of the inductive limit is analogous. Suppose we have $\rho,\sigma \in S$ with $\rho\leq \sigma$. Than we can form the commutative diagram
\begin{equation}\label{eq:prodc-corpodc-equivalence_diagram}
\begin{tikzcd}
\ell^1_*(\sigma_+)
\ar[d,"\pi_{\rho_+\sigma_+}"] \ar[r,"c_{\sigma_+}"]
\ar[rr,bend left,"\pi_{\sigma\sigma_+}"]
&
\ell^\infty(\sigma_+)
\ar[d, "\tilde\pi_{\rho_+\sigma_+}"] \ar[r,"\tau_\rho"]
\ar[rr,bend left,"\tilde\pi_{\sigma\sigma_+}"]
&
\ell^1_*(\sigma)
\ar[d,"\pi_{\rho\sigma}"] \ar[r,"c_{\sigma}"]
&
\ell^\infty(\sigma)
\ar[d,"\tilde\pi_{\rho\sigma}"]
\\
\ell^1_*(\rho_+)
\ar[r,"c_{\rho_+}"]
\ar[rr,bend right,"\pi_{\rho\rho_+}"]
&
\ell^\infty(\rho_+)
\ar[r,"\tau_\rho"]
\ar[rr,bend right,"\tilde\pi_{\rho\rho_+}"]
&
\ell^1_*(\rho)
\ar[r,"c_{\rho}"]
&
\ell^\infty(\rho)
\end{tikzcd}
\end{equation}
In the diagram \eqref{eq:prodc-corpodc-equivalence_diagram}, every morphism is the identity on elements. To be explicit,
\begin{itemize}
    \item [-] $\pi_{*\bullet}$ denotes the structure morphism of the projective system defining $\lambda(S)$;
    \item[-] $\tilde\pi_{*\bullet}$ denotes the structure morphism in the projective system of the $\ell^\infty$-variant of $\lambda(S)$;
    \item[-] $c_{\bullet}$ denotes the canonical morphism from the contracting coproduct to the contracting product.
\end{itemize}
The diagram \eqref{eq:prodc-corpodc-equivalence_diagram} shows that $(\tau_\rho)_{\rho\in S}$ is a morphism of projective systems and its composition with $(c_\rho)_{\rho\in S}$ produces $(\pi_{\rho\rho_+})_{\rho\in S}$ in one order and $(\tilde\pi_{\rho\rho_+})_{\rho\in S}$ in the other. We just need to check that the functor $\varprojlim_{\rho\in S}$ sends $(\pi_{\rho\rho_+})_{\rho\in S}$ to the identity of $\lambda(S)$ and $(\tilde\pi_{\rho\rho_+})_{\rho\in S}$ to the identity of $\varprojlim_{\rho\in S}\ell^\infty(\rho)$. This follows from the observation that $\rho\to\rho_+:S \to S$ is a cofinal injective map of partially ordered sets.

\end{proof}
In the case of Lemma \ref{lem:prodc_coprodc_exchange_in_lambda-kappa}, we have an isomorphism
\[
\kappa(S)^\vee\cong\lambda(S),
\]
but when we compute $\lambda(S)^\vee$ there is the problem that in general $(-)^\vee$ does not send limits to colimits. For this reason we study the mentioned problem in the case of some special Fréchet modules.
\subsection{Duality and Fréchet modules}
In what follows $F$ denotes a Fréchet $R$-module given by a projective system
\[
\pi_{ij}: F_j\to F_i,\quad i,j\in\N,\,i\leq j.
\]
The projection morphism $F\to F_j$ will be denoted by $\pi_j$.
\begin{definition}
A Fréchet module $F=\varprojlim_j F_j$ is said to be \emph{of type EM} if the projection morphisms
\[
\pi_j: F \to F_j
\]
are both monomorphisms and epimorphisms for all $j\in\N$.
\end{definition}

\begin{remark}\label{rmk:crit_for_being_EM}
To check that the morphism $\pi_j$ is a monomorphism and an epimorphism one can simply check that $\pi_j$ is injective and that there is a Banach sub-module $B\subset F$ such that $\pi_j(B)$ is dense in the Banach module $F_j$. This follows from the fact that $F$, as a bornological module, is a union of Banach modules and $\pi_j$, as a morphism from a bornological module to a Banach one, is determined by all the restrictions $\eval[1]{\pi_j}_B$ to the Banach sub-modules $B\subset F$. For $\pi_j$ to be a monomorphism it is sufficient and necessary that $\eval[1]{\pi_j}_B$ is injective for every Banach $B\subset F$. Since every element of $F$ is contained in some Banach sub-module, $\pi_j$ must be injective as a map of sets. For $\pi_j$ to be an epimorphism it is sufficient that $\eval[1]{\pi_j}_B$ is an epimorphism for some Banach $B\subset F$, which is true precisely when $\pi_j(B)$ is dense in $F_j$, for some Banach $B\subset F$.
\end{remark}

With the next proposition we introduce the main example we are interested in.

\begin{prop}\label{pr:EM-lambdas}
Let $\rho=(\rho_j)_{j\in\N}$ be a sequence of maps $\rho_j:\N\to\Rp$ such that $\rho_i<\rho_j$ for $i<j$. Then the sequential module
\[
\lambda(\rho):=
\varprojlim_{j\in\N}\coprodc_{n\in\N}[Re_n]_{\rho_j(n)}
\]
is a Fréchet module of type EM.
\end{prop}

\begin{proof}
By Remark \ref{rmk:crit_for_being_EM}, we have to check that, for all $j\in\N$, the projection morphism ${\pi_j:\lambda(\rho)\to\ell^1_*(\rho_j)}$ is injective and there is a Banach submodule $B$ of $F$ such that $\pi_j(B)$ is dense in $\ell^1_*(\rho_j)$.
The morphisms $\pi_j$ are injective because they act as the identity on elements. Recall that $\ell^1_*(\rho_j)$ is the completion of the algebraic $R$-submodule generated by $\{e_n:n\in\N\}$ for the norm induced by $\rho_j$. It is enough to prove that there is a Banach sub-module $B\subset\lambda(\rho)$ containing $e_n$ for every $n\in\N$. Define
\begin{equation*}
B=
\left\{
x=\sum_{n=0}^\infty x_n e_n \in \prod_{n\in\N} R\,e_n
: 
\norm{x}_B<\infty
\right\}
\end{equation*}
where $\norm{x}_B$ is defined to be the supremum
\begin{equation*}
\sup_{j\in\N}
\frac{\norm{x}_{\ell^1(\rho_j)}}{\sum_{k=0}^j\rho_j(k)}
=
\sup_{j\in\N}
\frac{\sum_{n=0}^\infty \abs{x_n}\rho_j(n)}{\sum_{k=0}^j\rho_j(k)}.
\end{equation*}
It is clear that $(B,\norm{-}_B)$ is a Banach module once we observe that it is the kernel of the morphism 
\begin{equation*}
\prodc_{j\in\N}
\left[
\ell^1_*(\rho_j)
\right]_{\frac{1}{\sum_{k=0}^j\rho_j(k)}}
\xrightarrow{\mkern10mu \Delta \mkern10mu}
\prodc_{j\in\N}
\left[
\ell^1_*(\rho_j)
\right]_{\frac{1}{\sum_{k=0}^j\rho_j(k)+\sum_{k=0}^{j+1}\rho_{j+1}(k)}},
\quad
(x^{(j)})_{j\in\N}
\;\longmapsto\;
(x^{(j)}-x^{(j+1)})_{j\in\N}
\end{equation*}
in $\Ban_R$. For every $n\in\N$, the element $e_n$ belongs to $B$ because we can compute its norm directly and see that it is finite:
\begin{equation*}
\norm{e_n}_B=\sup_{j\in\N} \frac{\rho_j(n)}{\sum_{k=0}^j\rho_j(k)}
\leq
1+\max_{j\leq n} \frac{\rho_j(n)}{\sum_{k=0}^j\rho_j(k)}
<\infty.
\end{equation*}
To end the proof we can check that the identity of $\prod_{n\in\N}R\,e_n$ induces a well-defined morphism from $\iota:B\to \lambda(\rho)$. Since
\[
\lambda(\rho)=\lim_{j\in\N}\ell^1_*(\rho_j),
\]
$\iota$ is well-defined if and only if all the morphisms
\begin{equation*}
\iota_j: B \to \ell^1_*(\rho_j), \quad x\mapsto x
\end{equation*}
are well-defined and bounded, but this is obvious from the definition of $\norm{x}_B$:
\begin{equation*}
\begin{split}
\norm{x}_{\ell^1_*(\rho_j)} 
&=
\sum_{n=0}^\infty \abs{x_n} \rho_j(n)
\\
&=
\sum_{k=0}^j \rho_j(k) 
\frac{\sum_{n=0}^\infty \abs{x_n} \rho_j(n)}{\sum_{k=0}^j \rho_j(k)}
\\
&\leq \Big( \sum_{k=0}^j \rho_j(k) \Big)\norm{x}_B.
\end{split}
\end{equation*}
To recap, have constructed a Banach module $B\subset \lambda(\rho)$ such that $\pi_j(B)$ contains $e_n$ for all $n\in\N$. Therefore $\pi_j(B)$ is dense in $\ell^1_*(\rho_j)$ for every $j\in\N$ and we can conclude that $\lambda(\rho)$ is of type EM.
\end{proof}

\begin{lemma} \label{lem:dual_FréchetEM}
Let $F$ be a Fréchet $R$-module of type EM. Let $(\pi_{ij}:F_j\to F_i)_{i\leq j}$ be the projective system defining $F$. Then, the natural morphism
\begin{equation*}
\varinjlim_{j\in\N} F_j^\vee \to F^\vee
\end{equation*}
induced by the inductive system $(\pi_{ij}^\vee: F_i^\vee \to F_j^\vee)_{i\leq j}$ is bijective as a map of sets.
\end{lemma}
\begin{proof}
Given any $j\in\N$, the assumption that $\pi_j : F \to F_j$ is a monomorphism implies that an element $\xi\in F_j^\vee$ is zero precisely when $\xi\circ\pi_j$ is. Hence, the map $\varinjlim_{j\in\N} F_j^\vee \to F^\vee$ is injective. Now we check that it is surjective.
Let $\xi:F\to R$ be an element of $F^\vee$. Consider the two cases:
\begin{enumerate}[a)]
    \item there is $j\in\N$ such that 
\begin{equation*}
\sup_{x\in F\setminus\{0\}}  \frac{\abs{\xi(x)}}{\norm{\pi_j(x)}_{F_j}}=:C <\infty;
\end{equation*}
    \item for every $j\in\N$
    \[
    \sup_{x\in F\setminus\{0\}}  
    \frac{\abs{\xi(x)}}{\norm{\pi_j(x)}_{F_j}}
    =\infty.
    \]
\end{enumerate}
If case a) holds, then we can extend $\xi$ to a morphism $F_j\to R$ in the following way. For any $x\in F_j$ there is a sequence $(x_n)_{n\in\N}$ of elements of $F$ such that $\norm{\pi_j(x_n)-x}_{F_j}$ converges to $0$ as $n\to\infty$ because $\pi_j(F)$ is a dense subset of the Banach module $F_j$. The sequence $(\xi(x_n))_{n\in\N}$ is Cauchy in $R$. Indeed, for every $\varepsilon>0$ there is $n(\varepsilon)\in\N$ such that $\norm{\pi_j(x_n)-x}_{F_j}\leq\varepsilon$ for any natural number $n\geq n(\varepsilon)$ and 
\[\begin{split}
\abs{\xi(x_n)-\xi(x_k)}
&=\abs{\xi(x_n-x_k)} \\
&\leq C \norm{\pi_j(x_n-x_k)}_{F_j} \\
&\leq C \Big(\norm{\pi_j(x_n)-x}_{F_j} + \norm{x-\pi_j(x_k)}_{F_j}\Big) \\
&\leq 2C\varepsilon
\end{split}\]
for all integers $n,k\geq n(\varepsilon)$. By definition of Banach rings, Cauchy sequences are convergent. Define
\[
\xi(x):=\lim_{n\to\infty} \xi(x_n).
\]
The value $\xi(x)$ does not depend on the choice of the sequence $(x_n)_{n\in\N}$, as any other sequence $(x_n')_{n\in\N}$ with $\pi_j(x_n')$ converging to $x$ satisfies
\[\begin{split}
\abs{\xi(x_n)-\xi(x_n')}
&=\abs{\xi(x_n-x_n')} \\
&\leq C \norm{\pi_j(x_n-x_n')}_{F_j} \\
&\leq C \norm{\pi_j(x_n)-\pi_j(x_n')}_{F_j} \\
&\leq C\Big(\norm{\pi_j(x_n)-x}_{F_j}+\norm{x-\pi_j(x_n')}_{F_j}\Big)
\xrightarrow{n\to\infty} 0.
\end{split}\]
By continuity of the norm we also have the inequality $\abs{\xi(x)}\leq C \norm{x}_{F_j}$. This definition gives the extension of $\xi$ to $F_j$. Now we prove by contradiction that the case b) is impossible. Suppose that 
\[
\sup_{x\in F\setminus\{0\}}  
\frac{\abs{\xi(x)}}{\norm{\pi_j(x)}_{F_j}}
=\infty
\]
for all $j\in\N$. Then, for every $j\in\N$ we can find $x_j\in F\setminus \{0\}$ such that 
\[
\abs{\xi(x_j)}\geq 2^j \norm{\pi_j(x_j)}_{F_j}.
\]
Define the sequence $\rho:\N\to\Rp$ by
\[
\rho(j)=\norm{\pi_j(x_j)}_{F_j},
\quad
\text{for all }j\in\N.
\]
Now we prove that there is a morphism 
\[
x: \ell^1_*(\rho) \to F
\]
such that $x(e_n)=x_n$ for every $n\in\N$. Note that 
$\Hom_R(\ell^1_*(\rho),F)$ consists of the $R$-linear maps 
$f:\ell^1_*(\rho)\to F$ which satisfy the condition
\[
\sup_{n\in\N}\frac{\norm{f(e_n)}_{F_j}}{\rho(n)}<\infty
\quad
\text{for every }j\in\N.
\]
We just need to verify that
\[
\sup_{n\in\N} \frac{\norm{x_n}_{F_j}}{\rho(n)}  < \infty
\]
for any $j\in\N$. For every $j\in\N$ and every integer $n\geq j$ we have
that $\norm{\pi_j(x_n)}_{F_j}\leq \rho(n)$ because 
\[\begin{split}
\norm{\pi_j(x_n)}_{F_j}
&= \norm{\pi_{jn}\big(\pi_n(x_n)\big)}_{F_j} \\
&\leq \norm{\pi_n(x_n)}_{F_n}\\
&=\rho(n).
\end{split}\]
Therefore,
\[
\sup_{n\in\N} \frac{\norm{\pi_j(x_n)}_{F_j}}{\rho(n)}
=\max_{n<j} \frac{\norm{\pi_j(x_n)}_{F_j}}{\rho(n)}
<\infty
\]
and the morphism $x:\ell^1_*(\rho) \to F$ is well-defined. Consider the composition $\xi\circ x$, which is a bounded $R$-linear map between Banach modules
\[
\xi\circ x : \ell^1_*(\rho)\to R,
\]
sending $e_n$ to $\xi(x_n)$. The boundedness of $\xi\circ x$ is equivalent to
\[
\sup_{n\in\N}\frac{\abs{\xi(x_n)}}{\rho(n)}
<\infty.
\]
However, by construction of the sequence $(x_n)_{n\in\N}$,
\[
2^n\leq\frac{\abs{\xi(x_n)}}{\rho(n)},
\quad
\text{for every }n\in\N,
\]
a clear contradiction.
\end{proof}

\begin{theorem}\label{th:iso_dual_Fréchet}
Let $F$ be a Fréchet $R$-module of type EM. Let $(\pi_{ij}:F_j\to F_i)_{i\leq j}$ be the projective system defining $F$. Then, the natural morphism
\begin{equation*}
\varinjlim_{j\in\N} F_j^\vee \to F^\vee
\end{equation*}
induced by the inductive system $(\pi_{ij}^\vee: F_i^\vee \to F_j^\vee)_{i\leq j}$ is an isomorphism in $\Ind(\Ban_R)$.    
\end{theorem}
\begin{proof}
Since objects of $\Ind(\Ban_R)$ are presheaves on $\Ban_R$, it is enough to check that for every Banach $R$-module $B$ the map
\begin{equation*}
\Hom_R\big( B,\varinjlim_{j\in\N} F_j^\vee \big)
\to 
\Hom_R(B,F^\vee)
\end{equation*}
is bijective. Consider a morphism of flat projective Banach $R$-modules $P_1\to P_0$ such that $B$ is its cokernel as in Proposition \ref{pr:proj_resol_of_Banach}.
Consider the commutative square 
\begin{equation}\label{eq:square1_main_th}
\begin{tikzcd}
\Hom_R(P_0,\varinjlim_{j\in\N} F_j^\vee) 
\ar[r] \ar[d]
&
\Hom_R(P_0,F^\vee) 
\ar[d]
\\
\Hom_R(P_1,\varinjlim_{j\in\N} F_j^\vee) 
\ar[r]
&
\Hom_R(P_1,F^\vee)
\end{tikzcd}
\end{equation}
induced by the morphism $P_1\to P_0$ and the natural morphism $\varinjlim_j F_j^\vee \to F^\vee$. It is naturally isomorphic to the commutative square
\begin{equation}\label{eq:square2_main_th}
\begin{tikzcd}
\varinjlim_{j\in\N}\Hom_R(P_0\wot F_j, R)
\ar[r] \ar[d]
&
\Hom_R(\varprojlim_{j\in\N} P_0\wot F_j, R)
\ar[d]
\\
\varinjlim_{j\in\N}\Hom_R(P_1\wot F_j, R)
\ar[r]
&
\Hom_R(\varprojlim_{j\in\N} P_1\wot F_j, R)
\end{tikzcd}
\end{equation}
for the following reason. Consider first the left column of the squares \eqref{eq:square1_main_th} and \eqref{eq:square2_main_th}. For $k=0,1$, we have the natural isomorphism
\begin{equation*}
\Hom_R(P_k,\varinjlim_{j\in\N} F_j^\vee)
\cong
\varinjlim_{j\in\N}\Hom_R(P_k, F_j^\vee)
\end{equation*}
because $P_k$ is Banach and 
\[
\varinjlim_{j\in\N} F_j^\vee
=
\indlim_{j\in\N}F_j^\vee
\] 
in $\Ind(\Ban_R)$. For every $j\in\N$, tensor-hom adjunction gives the natural isomorphism
\[
\Hom_R(P_k,F_j^\vee)
\cong
\Hom_R(P_k\wot F_j, R).
\]
On the right column of \eqref{eq:square1_main_th} and \eqref{eq:square2_main_th}, the same adjunction produces the natural isomorphism
\[
\Hom_R(P_k,F^\vee)
\cong
\Hom_R(P_k\wot F, R).
\]
Recall that $P_k$ is a flat projective Banach module. Thus, by Proposition \ref{pr:some_metrizable_mods}, it is flat and metrizable. By Proposition \ref{pr:lim-wot-compatibility_for_metrizables}, $P_k\wot-$ commutes with countable projective limits, in particular,
\[
P_k\wot F \cong \varprojlim_{j\in\N} P_k\wot F_j
\]
naturally. Now observe that the rows of the commutative square \eqref{eq:square2_main_th} are the underlying maps of sets of the natural morphism connecting the dual of a Fréchet module with the colimit of its dual inductive system as in Lemma \ref{lem:dual_FréchetEM}. The Fréchet module 
\[
P_k\wot F=\varprojlim_{j\in\N} P_k \wot F_j
\]
is of type EM because the projection $\id_{P_k}\wot \pi_j:P_k\wot F \to P_k\wot F_j$ is obtained by tensoring $\pi_j$ with $P_k$. Since $P_k$ is flat in the sense of Definition \ref{def:flat_object}, the functor $P_k\wot-$ sends monomorphisms to monomorphisms. The functor $P_k\wot-$ also preserves every epimorphisms because it has a right adjoint. This ensures that $\id_{P_k}\wot \pi_j$ is a monomorphism and an epimorphism. By Lemma \ref{lem:dual_FréchetEM}, the rows of the commutative squares \eqref{eq:square1_main_th} and \eqref{eq:square2_main_th} are bijective. Since $B=\Coker(P_1\to P_0)$, we have a commutative diagram 
\begin{equation*}
\begin{tikzcd}
0 \ar[r]
&
\Hom_R(B,\varinjlim_{j\in\N} F_j^\vee)
\ar[r] \ar[d] 
&
\Hom_R(P_0,\varinjlim_{j\in\N} F_j^\vee)
\ar[r] \ar[d] 
&
\Hom_R(P_1,\varinjlim_{j\in\N} F_j^\vee)
\ar[d] 
\\
0 \ar[r]
&
\Hom_R(B, F^\vee)
\ar[r]
&
\Hom_R(P_0,F^\vee)
\ar[r] 
&
\Hom_R(P_1, F^\vee)
\end{tikzcd}
\end{equation*}
with exact rows where the two vertical arrows on the right are bijective for what we observed so far. Therefore, by exactness, also the vertical arrow on the left
\begin{equation*}
\Hom_R\big( B,\varinjlim_{j\in\N} F_j^\vee \big)
\to 
\Hom_R(B,F^\vee)
\end{equation*}
is bijective. This concludes the proof.
\end{proof}

\begin{remark}
If we replace the functor $(-)^\vee=\ihom_R(-,R)$ in Lemma \ref{lem:dual_FréchetEM} and Theorem \ref{th:iso_dual_Fréchet} with $\ihom_R(-,M)$, where $M$ is any Banach $R$-module, the same proofs still work. In this way we get that a canonical isomorphism
\begin{equation*}
\varinjlim_{j\in\N} \ihom_R(F_j,M) \cong \ihom_R(F,M)
\end{equation*}
for any Banach $R$-module $M$ and any Fréchet $R$-module $F$ of type EM.
\end{remark}

\subsection{Duality for sequential modules}

Now we prove that many sequential modules are reflexive and the duality interchange sequential modules of type $\kappa$ with sequential modules of type $\lambda$.

\begin{theorem}\label{th:duality_lambda-kappa}
Let $\rho:\N\times \N \to \Rp\,\ (j,n)\mapsto \rho_j(n)$ be a function. Suppose that 
\begin{gather*}
\sum_{n=0}^\infty \frac{\rho_j(n)}{\rho_{j+1}(n)}<\infty
\\
\Big(\, \lim_{n\to\infty} \frac{\rho_j(n)}{\rho_{j+1}(n)} = 0
\quad \textrm{in the non-archimedean case.}\, \Big)
\end{gather*}
Then, the modules 
\begin{align*}
\lambda(\rho)&:=\varprojlim_{j\in\N}\ell^1_*(\rho_j),
\\
\kappa(\rho)&:=\varinjlim_{j\in\N}\ell^1_*(\rho_j^{-1})
\end{align*}
are dual to each other. Namely, there is a pairing 
\begin{equation*}
    \kappa(\rho) \wot \lambda(\rho) \to R
\end{equation*}
such that the induced canonical morphisms
\begin{align*}
&\lambda(\rho) \to \kappa(\rho)^\vee, \\
&\kappa(\rho) \to \lambda(\rho)^\vee
\end{align*}
are isomorphisms. In particular $\lambda(\rho)$ and $\kappa(\rho)$ are reflexive bornological $R$-modules.
\end{theorem}

\begin{proof}
By Lemma \ref{lem:prodc_coprodc_exchange_in_lambda-kappa}, we may replace contracting coproducts with contracting products in the limit defining $\lambda(\rho)$, so that
\begin{equation*}
\lambda(\rho)
\cong
\varprojlim_{j\in\N}\ell^\infty(\rho_j).
\end{equation*}
Also the module $\kappa(\rho)^\vee$ is a projective limit of the form
\begin{equation*}
\kappa(\rho)^\vee
\cong 
\varprojlim_{j\in\N} \ell^1_*(\rho_j^{-1})^\vee
\end{equation*}
because the functor $(-)^\vee:\Ind(\Ban_R)^{\mathrm{op}} \to \Ind(\Ban_R)$ commutes with limits. For every $j\in\N$ we have the isometric isomorphism
\begin{equation*}
\psi_j : 
\prodc_{n\in\N} [Re_n]_{\rho_j(n)} 
\to 
\Big(\coprodc_{n\in\N} [Re_n]_{\rho_j(n)^{-1}}\Big)^\vee,
\quad
\sum_{n=0}^\infty x_n e_n 
\mapsto
\sum_{n=0}^\infty x_n \check{e}_n,
\end{equation*}
where $\check{e}_n:\ell^1(\rho_j^{-1})\to R$ is the projection to the $n$-th component. Since $\psi_j$ is isometric, it is easy to see that the collection $(\psi_j)_{j\in\N}$ defines an isomorphism of projective systems. If we take the limit, we obtain an isomorphism
\begin{equation*}
\psi:\lambda(\rho) \to \kappa(\rho)^\vee,
\quad
\sum_{n=0}^\infty x_n e_n \mapsto \sum_{n=0}^\infty x_n \check{e}_n .
\end{equation*}
The composition of $\psi\wot \id_{\kappa(\rho)}$ with the canonical pairing $\kappa(\rho)^\vee\wot \kappa(\rho) \to R$ produces a pairing
\begin{equation*}
g:\lambda(\rho)\wot \kappa(\rho) \to R. 
\end{equation*}
If $x\in\lambda(\rho)$ is equal to $\sum_n x_n e_n$ and $y\in\kappa(\rho)$ is given by the sum $\sum_n y_n e_n$, then, by the definition of $\psi$,
\begin{equation}\label{eq:pairing_lambda-kappa}
g(x \otimes y) = \sum_{n\in\N} x_n y_n.
\end{equation}
There is an adjoint $\psi^a:\kappa(\rho) \to \lambda(\rho)^\vee$, where $\psi^a$ is the composition of the canonical morphism $\kappa(\rho)\to\kappa(\rho)^{\vee\vee}$ with the dual $\psi^\vee$. As a map of sets, it is uniquely defined by the relation
\[
\psi^a(y)(x)=g(x\otimes y)=\psi(x)(y)
\]
for every $x\in\lambda(\rho)$ and $y\in \kappa(\rho)$. Now we verify that $\psi^a$ is an isomorphism. By Lemma \ref{lem:prodc_coprodc_exchange_in_lambda-kappa}, we can interchange contracting coproducts and contracting products in the colimit defining $\kappa(\rho)$:
\[
\kappa(\rho)
\cong
\varinjlim_{j\in\N} \ell^\infty(\rho_j^{-1}). 
\]
Recall from Proposition \ref{pr:EM-lambdas} that $\lambda(\rho)$ is a Fréchet module of type EM. Then, by Theorem \ref{th:iso_dual_Fréchet}, the dual of the limit becomes the colimit of the duals. In other terms, the canonical morphism
\[
\varinjlim_{j\in\N} \ell^1_*(\rho_j)^\vee 
\to
\lambda(\rho)^\vee
\]
is an isomorphism. For every $j\in\N$, consider the isometric isomorphism
\begin{equation*}
\psi^a_j : 
\prodc_{n\in\N} [Re_n]_{\rho_j(n)^{-1}} 
\to 
\Big(\coprodc_{n\in\N} [Re_n]_{\rho_j(n)}\Big)^\vee,
\quad
\sum_{n=0}^\infty x_n e_n 
\mapsto
\sum_{n=0}^\infty x_n \check{e}_n,
\end{equation*}
As before, the collection $(\psi^a_j)_{j\in\N}$ defines an isomorphism of inductive systems
\[
\varinjlim_j\psi^a_j : 
\varinjlim_{j\in\N} \ell^\infty(\rho_j^{-1})
 \to
\varinjlim_{j\in\N} \ell^1_*(\rho_j)^\vee.
\]
By checking what the morphisms do on elements, it's easy to see that the square 
\[\begin{tikzcd}
\kappa(\rho)
\ar[r,"\psi^a"]
&
\lambda(\rho)^\vee
\\
\varinjlim_{j\in\N} \ell^\infty(\rho_j^{-1})
\ar[r,"\varinjlim_j\psi^a_j"]
\ar[u,sloped,"\sim"]
&
\varinjlim_{j\in\N} \ell^1_*(\rho_j)^\vee
\ar[u,sloped,"\sim"]
\end{tikzcd}\]
commutes, so $\psi^a$ must be an isomorphism.
\end{proof}

\subsection{Duality for countable products and coproducts}
In the previous part, we proved reflexivity for modules of sequences with growth conditions. Here we show that even the modules
\begin{equation*}
\coprod_{n\in\N} R\, e_n
\quad \text{and} \quad
\prod_{n\in\N} R\, e_n
\end{equation*}
are dual to each other.

\begin{definition}\label{def:directed_aleph_1_filtered_set_of_sequences}
For any set $I$ let $\Upsilon(I)$ be the directed ordered set consisting of functions ${\varphi:I \to \Np}$ with the order relation
\[
\varphi_1<\varphi_2 
\quad  \text{if there is a finite $J\subset I$ such that}\quad
\varphi_1(i)<\varphi_2(i)
\quad \text{for all }i\in I\setminus J.
\]
\end{definition}

\begin{lemma}
Let $I$ be a set and $\{M_i:i\in I\}$ a family of Banach $R$-modules. Let $\varphi_1,\varphi_2\in\Upsilon(I)$ be such that $\varphi_1<\varphi_2$. Then the identity map induces well-defined morphisms
\begin{align*}
&\prodc_{i\in I} [M_i]_{\varphi_1(i)^{-1}}
\to
\prodc_{i\in I} [M_i]_{\varphi_2(i)^{-1}},
\\
&\prodc_{i\in I} [M_i]_{\varphi_2(i)}
\to
\prodc_{i\in I} [M_i]_{\varphi_1(i)}.
\end{align*}
\end{lemma}

\begin{proof}
Define the constant
\[
C:=\sup_{i\in I} \frac{\varphi_1(i)}{\varphi_2(i)}
\]
which is a finite positive real number because $\varphi_1<\varphi_2$. For any sequence $m=(m_i)_{i\in I}\in M^I$, we have
\[\begin{split}
\sup_{i\in I} \frac{\norm{m_i}}{\varphi_2(i)} 
&=
\sup_{i\in I} 
\frac{\norm{m_i}}{\varphi_1(i)} \frac{\varphi_1(i)}{\varphi_2(i)}
\\
&\leq
C\cdot \sup_{i\in I} \frac{\norm{m_i}}{\varphi_1(i)} 
\end{split}\]
and 
\[\begin{split}
\sup_{i\in I}\ \norm{m_i}\varphi_1(i)
&=
\sup_{i\in I}\ 
\norm{m_i}\varphi_2(i) \frac{\varphi_1(i)}{\varphi_2(i)}
\\
&\leq
C\cdot \sup_{i\in I}\ \norm{m_i}\varphi_2(i).
\end{split}\]
This proves the claim.
\end{proof}

\begin{prop}\label{pr:prod_of_Banach_is_colimit}
Let $I$ be a set and $\{M_i:i\in I\}$ a family of Banach $R$-modules.
Then,
\begin{equation*}
\prod_{i\in I} M_i
=
\varinjlim_{\varphi\in\Upsilon(I)} \prodc_{i\in I} [M_i]_{\varphi(i)^{-1}}.
\end{equation*}
\end{prop}

\begin{proof}
This is proved in \cite[Lemma 5.11]{fremod}.
\end{proof}

There is also a dual version of the previous result for the coproduct.

\begin{prop}\label{pr:coprod_of_Banach_is_limit}
Let $\{M_i:i\in I\}$ be a family of Banach $R$-modules. Then the canonical morphism
\begin{equation}\label{eq:morphism_coprod_to_limit}
\coprod_{i\in I}M_i
\to
\varprojlim_{\varphi\in\Upsilon(I)}\prodc_{i \in I}[M_i]_{\varphi(i)}
\end{equation}
is an isomorphism in $\Ind(\Ban_R)$.
\end{prop}

\begin{proof}
It is enough to prove that the induced map on hom-sets
\begin{equation}\label{eq:hom_morphism_coprod_to_limit}
\Hom_R\Big(N,\,\coprod_{i\in I}M_i\Big)\longrightarrow
\Hom_R\Big(N,\varprojlim_{\varphi\in\Upsilon(I)}\prodc_{i \in I}[M_i]_{\varphi(i)}\Big)
\end{equation}
is bijective for every Banach module $N$. It is injective because the morphism \eqref{eq:morphism_coprod_to_limit} is a monomorphism. We now check that it is surjective. Let 
\[
f:N\to\varprojlim_{\varphi\in\Upsilon(I)}\prodc_{i \in I}[M_i]_{\varphi(i)}
\]
be a morphism. By definition of the limit it is determined by the family of morphisms
\[
f^\varphi: N\to\prodc_{i \in I}[M_i]_{\varphi(i)}
\]
given by composing $f$ with the structure morphisms of the limit. Proving that the map \eqref{eq:hom_morphism_coprod_to_limit} is surjective amounts to proving that $f$ factors as a composition of the form
\begin{equation}\label{eq:cd_f_factors_lem:coprod_of_Banach_is_limit}
\begin{tikzcd}[row sep=small]
    & \coprod_{i\in J }M_i \ar[dd]\\
    N \ar[ur, dashed] \ar[dr,"f"'] & \\
    & \displaystyle\varprojlim_{\varphi\in\Upsilon(I)}\prodc_{i \in I}[M_i]_{\varphi(i)}
\end{tikzcd}
\end{equation}
for some finite subset $J$ of $I$. For every $i\in I$, let $f^\varphi_i$ be composition of $f^\varphi$ with the projection to $[M_i]_{\varphi(i)}$. In this way the operator norm of $f^\varphi$ can be computed by
\[
\norm{f^\varphi}=\sup_{i\in I}\ \norm{f_i^\varphi}.
\]
because $\ihom_R(N,-)$ commutes with contracting products (Equation \eqref{eq:interaction_contr_(co)prod_and_monoidal_str}).
Note that for all $x\in N$ the element $f_i^\varphi(x)\in M_i$ does not depend on $\varphi$ because the morphisms in the projective systems are the identity on elements. This means that there is a unique morphism ${f_i:N\to M_i}$ such that $f_i(x)=f_i^\varphi(x)$ for all $x\in N$ and all $\varphi\in\Upsilon(I)$. In particular 
\[
\norm{f_i^\varphi}=\norm{f_i}\varphi(i)
\]
where $\norm{f_i}$ does not depend on $\varphi$. Therefore
\[
\norm{f^\varphi}=\sup_{i\in I}\ \norm{f_i}\varphi(i).
\]
for all $\varphi\in\Upsilon(I)$.
If the set $J=\{i\in I: f_i\neq0\}$ is finite then $f$ comes from a morphism $N\to\coprod_{i\in J} M_i$ which is what we want. Suppose by contradiction that $\{i\in I: f_i\neq0\}$ is not finite. Then there exists a countable subset $\{i_k:k\in\N\}\subseteq I$ such that $\norm{f_{i_k}}\neq0$ for all $k\in\N$. Define a $\varphi$ as follows:
\[
\varphi(i)=\begin{cases}
    1 & \text{if $i\neq i_k$ for all $k\in\N$}, \\
    \min\{n\in\N : n\norm{f_{i_k}}>k\} & \text{if }i=i_k.
\end{cases}
\]
for every $i\in I$.
The morphism $f^\varphi$ is bounded, but
\[\begin{split}
\norm{f^\varphi}
&=
\sup_{i \in I}\ \norm{f_i}\varphi(i)
\\
&\geq
\sup_{k\in\N}\ \norm{f_{i_k}}\varphi(i_k)
\\
&\geq
\sup_{k\in\N}\ k = +\infty,
\end{split}\]
a contradiction. Therefore $\{i\in I : f_i\neq0\}$ must be finite.
\end{proof}

\begin{corollary}\label{cor:duality_prod_coprod}
The bornological $R$-modules $\prod_{i\in\N} R\, e_i$ and $\coprod_{i\in\N} R\,e_i$ are dual to each other. In more precise terms, the pairing
\[
\coprod_{i\in \N} R\,e_i \wot \prod_{i\in \N} R\, e_i 
\to R,
\quad
\sum_{0\leq i\ll\infty}\!a_i e_i\ \otimes\ \sum_{i\in\N} b_i e_i
\ \mapsto
\sum_{0\leq i \ll\infty}\! a_i b_i
\]
induces isomorphisms
\[
\coprod_{i\in \N} R\,e_i \cong \Big( \prod_{i\in \N} R\,e_i \Big)^\vee,
\qquad
\prod_{i\in \N} R\,e_i \cong \Big( \coprod_{i\in \N} R\,e_i \Big)^\vee.
\]
\end{corollary} 

\begin{proof}
Since $\ihom_R(-,R)$ has a left adjoint, it is automatic that 
\[
\prod_{i\in \N} R\,e_i \cong \Big( \coprod_{i\in \N} R\,e_i \Big)^\vee,
\]
and the isomorphism is equivalent by tensor-hom adjunction to the pairing given in the statement. We have to check that the dual of the product is the coproduct.
Let $\Upsilon$ be the set $\Upsilon(\N)$ as defined in \ref{def:directed_aleph_1_filtered_set_of_sequences}. By the two propositions \ref{pr:prod_of_Banach_is_colimit} and \ref{pr:coprod_of_Banach_is_limit} we have that
\begin{equation}\label{eq:formula1_in_proof_dual_of_prod}
\prod_{i\in \N} R\,e_i 
= 
\varinjlim_{\varphi\in\Upsilon} \prodc_{i\in \N} [R\,e_i]_{\varphi(i)^{-1}}
=
\varinjlim_{\varphi\in\Upsilon} \ell^\infty(\varphi^{-1})
\end{equation}
and
\[
\coprod_{i\in \N} R\,e_i
=
\varprojlim_{\varphi\in\Upsilon}\prodc_{i \in \N}[R\,e_i]_{\varphi(i)}
=
\varprojlim_{\varphi\in\Upsilon}\ell^\infty(\varphi).
\]
For every $\varphi\in\Upsilon$ define $\varphi_+$ as the sequence
\[
\varphi_+(i)=2^i\varphi(i),
\quad\text{for all }i\in\N.
\]
The set of sequences $\Upsilon$ satisfies the hypotheses of Lemma \ref{lem:prodc_coprodc_exchange_in_lambda-kappa} because $\varphi(i)\leq2^i\varphi(i)$ for all $i\in\N$ and
\[
\sum_{i\in\N} \frac{\varphi(i)}{\varphi_+(i)}=2<\infty.
\]
Thus, we can exchange the contracting product with the contracting coproduct in the equation \eqref{eq:formula1_in_proof_dual_of_prod}. To conclude, recall that $\ihom_R(-,R)$ sends colimits to limits and contracting coproducts to contracting products:
\[\begin{split}
\Big( \prod_{i\in \N} R\,e_i \Big)^\vee
&\cong
\ihom_R\Big( 
\varinjlim_{\varphi\in\Upsilon} \coprodc_{i\in \N} [R\,e_i]_{\varphi(i)^{-1}}
,R\Big)
\\
&\cong
\varprojlim_{\varphi\in\Upsilon} \ihom_R\Big( 
\coprodc_{i\in \N} [R\,e_i]_{\varphi(i)^{-1}}
,R\Big)
\\
&\cong
\varprojlim_{\varphi\in\Upsilon}
\prodc_{i \in \N}[R\,\check{e}_i]_{\varphi(i)}
\\
&\cong 
\coprod_{i\in \N} R\,\check{e}_i.
\end{split}\]
Here $\check{e}_i$ denotes the projection from the product to $R\,e_i$.
\end{proof}

\subsection{Nuclearity of the sequential modules}

\begin{theorem}\label{th:nuclearity_kappa-lambda}
Let $\rho:\N\times\N \to \Rp,\; (j,n)\mapsto \rho(j,n)=\rho_j(n)$ be a function satisfying 
\begin{equation}\label{eq:nuclear_condition_on_matrix_rho}
\sum_{n=0}^\infty \frac{\rho(j,n)}{\rho(j+1,n)} < \infty,
\quad
\text{for all $j\in\N$.}
\end{equation}
Then, the bornological $R$-modules
\begin{equation*}
\kappa(\rho)
\quad \text{and} \quad
\lambda(\rho)
\end{equation*}
are nuclear.
\end{theorem}

\begin{proof}
The bornological module $\kappa:=\kappa(\rho)$ is the filtered colimit of the diagram
\[
\ell^1_*(\rho_0^{-1}) \to
\cdots \to 
\ell^1_*(\rho_j^{-1}) \to \ell^1_*(\rho_{j+1}^{-1}) 
\to \cdots.
\]
Under the condition \eqref{eq:nuclear_condition_on_matrix_rho}, the transition morphisms of the above diagram are nuclear by Proposition \ref{pr:nuclear_morph_ell}. Therefore, $\kappa$ is nuclear.

Now we prove that $\lambda:=\lambda(\rho)$ is nuclear. As an object of the category $\Ind(\Ban_R)$ it is represented by the filtered colimit
\[
\indlim_{B\in\Ban_R/\lambda}\!\! B.
\]
To prove that $\lambda$ is nuclear, it is sufficient to show that for every $B\in\Ban_R$ and every morphism $f:B\to\lambda$, there is a commutative diagram
\begin{equation}\label{eq:th:nuclearity_kappa-lambda:triangle1}
\begin{tikzcd}[row sep=small]
B' \ar[rrd] & & 
\\ 
& & \lambda 
\\
B \arrow[uu,"\tilde{f}"] \arrow[rru,"f"'] & & \\
\end{tikzcd}
\end{equation}
with $B'\in\Ban_R$ and $\tilde{f}$ nuclear. The morphism $\tilde{f}$ is produced in two steps. First, construct a commutative triangle
\begin{equation}\label{eq:th:nuclearity_kappa-lambda:triangle2}
\begin{tikzcd}
\ell^\infty(\sigma^{-1}) \ar[rr,"\iota_\sigma"] & & \lambda \\
B \arrow[u,"f^\sigma"] \arrow[rru,"f"'] & & \\
\end{tikzcd}
\end{equation}
for a sequence $\sigma:\N\to\Rp$, where $\iota_\sigma$ is the identity on elements. Then, construct another commutative triangle
\begin{equation}\label{eq:th:nuclearity_kappa-lambda:triangle3}
\begin{tikzcd}
\ell^\infty(\sigma_+^{-1}) \ar[rrd,"\iota_{\sigma_+}"] & & \\
\ell^\infty(\sigma^{-1}) 
\arrow[u,"\iota_{\sigma\sigma_+}"] \arrow[rr,"\iota_\sigma"'] 
& & \lambda 
\end{tikzcd}
\end{equation}
where all morphisms act as the identity on elements and $\sigma_+:\N\to \Rp$ is such that $\iota_{\sigma\sigma_+}$ is nuclear. Finally, combine the two commutative triangles 
\eqref{eq:th:nuclearity_kappa-lambda:triangle2} and 
\eqref{eq:th:nuclearity_kappa-lambda:triangle3} 
to produce the nuclear morphism $\tilde{f}:B \to \ell^\infty(\sigma_+^{-1})$ as the composition $\iota_{\sigma\sigma_+}\circ f^{\sigma}$. 
Recall that, by Lemma \ref{lem:prodc_coprodc_exchange_in_lambda-kappa}, 
\[
\lambda=\varprojlim_{j\in\N}\ell^\infty(\rho_j).
\]
In this manner, the collection of morphisms $B \to \lambda$ can be determined as follows:
\begin{equation}\label{eq:th:nuclearity_kappa-lambda:1}
\begin{aligned}
\ihom_R(B,\lambda)
&\cong
\varprojlim_{j\in\N} \ihom_R\big(B,\ell^\infty(\rho_j)\big)
\\
&\overset{\eqref{eq:interaction_contr_(co)prod_and_monoidal_str}}{\cong}
\varprojlim_{j\in\N} \prodc_{n\in\N}[B^\vee]_{\rho_j(n)}.
\end{aligned}
\end{equation}
The combined isomorphisms \eqref{eq:th:nuclearity_kappa-lambda:1} send a morphism ${f:B \to \lambda}$ to the sequence of morphisms $f_n: B \to R$ satisfying the conditions
\begin{gather*}
\sup_{n\in\N}\ \norm{f_n}_{B^\vee}\rho(j,n) <\infty, \qquad
\text{for all $j\in\N$,}
\\
f(x)=\sum_{n=0}^\infty f_n(x) \, e_n, \qquad \text{for all $x\in B$}
\end{gather*}
Note that the sequence $(f_n)_{n\in\N}$ automatically satisfies the stronger condition
\begin{equation*}
\sum_{n=0}^\infty \norm{f_n}_{B^\vee} \rho(j,n) <\infty,
\qquad \text{for all $j\in\N$,}
\end{equation*}
because 
\begin{equation}\label{eq:th:nuclearity_kappa-lambda:2}
\begin{split}
\sum_{n=0}^\infty \norm{f_n}_{B^\vee} \rho(j,n)
&= \sum_{n=0}^\infty \norm{f_n}_{B^\vee}\rho(j+1,n) \frac{\rho(j,n)}{\rho(j+1,n)}
\\
&\leq \Big(\sup_{n\in\N}\ \norm{f_n}_{B^\vee}\rho(j+1,n) \Big) 
\sum_{n=0}^\infty \frac{\rho(j,n)}{\rho(j+1,n)}.
\end{split}
\end{equation}
for every $j\in\N$. If we apply the isomorphisms of \eqref{eq:th:nuclearity_kappa-lambda:1} to the specific situation where $B$ is a Banach $R$-module of the form $\ell^\infty(\sigma^{-1})$, for some sequence $\sigma:\N \to \R_p$, then we obtain that the identity induces a morphism $\iota_\sigma\ell^\infty(\sigma^{-1})\to\lambda$ if and only if 
\[
\sup_{n\in\N}\ \sigma(n)\rho(j,n)<\infty, \qquad
\text{for all $j\in\N$.}
\]
Let $\Upsilon$ be the set of such sequences. By the same argument of \eqref{eq:th:nuclearity_kappa-lambda:2}, $\sigma\in\Upsilon$ if and only if
\[
\sum_{n=0}^\infty \sigma(n)\rho(j,n),
\qquad \text{for all $j\in\N$.}
\]
Fix a morphism $f:B \to \lambda$ with its corresponding sequence $(f_n)_{n\in\N}$ under \eqref{eq:th:nuclearity_kappa-lambda:1}. To produce a morphism $f^\sigma: B \to \ell^\infty(\sigma^{-1})$ as in the diagram \eqref{eq:th:nuclearity_kappa-lambda:triangle2}, it is necessary and sufficient that there is a $\sigma\in\Upsilon$ such that $\norm{f_n}_{B^\vee}$ is bounded by $\sigma(n)$ uniformly in $n\in\N$. We can not take a priori $\sigma(n):=\norm{f_n}_{B^\vee}$ because it can be that $\norm{f_n}_{B^\vee}=0$ for some $n\in\N$. Thus, we need to check first that $\Upsilon$ is non-empty, so that we can find a $\sigma_0\in\Upsilon$ and set
\[
\sigma(n)=\begin{cases}
    \norm{f_n}_{B^\vee} & \text{if }f_n\neq0,\\
    \sigma_0(n) & \text{otherwise}.
\end{cases}
\]
Define
\[
r_j:=\sup_{n\in\N} \frac{\rho(j,n)}{\rho(j+1,n)}
\]
for all $j\in\N$. It is finite by the assumption \eqref{eq:nuclear_condition_on_matrix_rho}, and it is non-zero because it is the supremum of a family of non-zero real numbers. Let $\sigma_0:\N \to \Rp$ be defined by
\[
\sigma_0(n)=\rho(n,n)^{-1}\prod_{i=0}^{n-1} r_i^{-1}
\]
for all $n\in\Np$ and $\sigma_0(n)=1$ for $n=0$. We claim that $\sigma_0\in\Upsilon$. For all $j\in\N$ and every integer $n>j$, we can rewrite $\sigma_0(n)\sigma(j,n)$ as the product
\[
\left( \prod_{i=0}^j r_i^{-1} \right)
\left( \prod_{i=j}^{n-1} r_i^{-1}\frac{\rho(i,n)}{\rho(i+1,n)} \right),
\]
where the product from $0$ to $j-1$ is assumed to equal $1$ when $j=0$. The term $r_i^{-1}\rho(i,n)\rho(i+1,n)^{-1}$ is not greater than $1$ for all $i=j,\dots n-1$ by construction. Therefore,
\[
\sigma_0(n)\sigma(j,n) \leq \prod_{i=0}^{j-1} r_i^{-1}
\]
for all integers $n>j$. In particular,
\[
\sup_{n\in\N}\ \sigma_0(n)\rho(j,n) < \infty
\]
for all $j\in\N$, so that $\sigma_0\in\Upsilon$ as claimed. So far, we have a morphism $f^\sigma: B \to \ell^\infty(\sigma^{-1})$ fitting in the commutative triangle \eqref{eq:th:nuclearity_kappa-lambda:triangle2}. By Proposition \ref{pr:nuclear_morph_ell}, in order to have the required nuclear morphism $\iota_{\sigma\sigma_+}$ in the commutative triangle \eqref{eq:th:nuclearity_kappa-lambda:triangle3}, it is sufficient that, given $\sigma\in\Upsilon$, there is a sequence $\sigma_+\in\Upsilon$ such that
\begin{equation}\label{eq:th:nuclearity_kappa-lambda:3}
\sum_{n=0}^\infty \frac{\sigma(n)}{\sigma_+(n)}<\infty.
\end{equation}
Define 
\[
s_j=\sum_{n=0} \frac{\rho(j,n)}{\rho(j+1,n)}
\]
for all $j\in\N$, and $\gamma:\N \to \Rp$ by 
\[
\gamma(n)=\sum_{i=0}^\infty \frac{\rho(i,n)}{2^is_i\rho(i+1,n)}
\]
for all $n\in\N$. We claim that the sequence $\sigma_+$ defined by 
\[
\sigma_+(n)=\frac{\sigma(n)}{\gamma(n)},
\]
for every $n\in\N$, it satisfies the required properties. We first check that $\sigma_+$ belongs to $\Upsilon$. Fix an arbitrary $j\in\N$. We have the estimate
\begin{equation}\label{eq:th:nuclearity_kappa-lambda:estimate1}
\frac{1}{\gamma(n)} \leq \frac{2^js_j\rho(j+1,n)}{\rho(j,n)}
\end{equation}
for all $n\in\N$. Hence,
\[\begin{split}
\sum_{n=0}^\infty \sigma_+(n) \rho(j,n) 
&= \sum_{n=0}^\infty \sigma(n) \frac{\rho(j,n)}{\gamma(n)}
\\
&\overset{\eqref{eq:th:nuclearity_kappa-lambda:estimate1}}{\leq}
\sum_{n=0}^\infty \sigma(n) \rho(j,n)\frac{2^js_j\rho(j+1,n)}{\rho(j,n)}
\\
&= 2^js_j \sum_{n=0}^\infty \sigma(n) \rho(j+1,n) < \infty.
\end{split}\]
Next, we verify that \eqref{eq:th:nuclearity_kappa-lambda:3} holds. By the definition of $\sigma_+$, this amounts to showing that
\begin{equation}\label{eq:th:nuclearity_kappa-lambda:4}
\sum_{n=0}^\infty \gamma(n)
=\sum_{n=0}^\infty \sum_{i=0}^\infty 
\frac{\rho(i,n)}{2^is_i\rho(i+1,n)}.
\end{equation}
Each term of the double sum is a strictly positive real number. Hence, we can interchange the order of summation in \eqref{eq:th:nuclearity_kappa-lambda:4}, obtaining
\[\begin{split}
\sum_{n=0}^\infty \gamma(n)
&= \sum_{i=0}^\infty \sum_{n=0}^\infty 
\frac{\rho(i,n)}{2^is_i\rho(i+1,n)}
\\
&= \sum_{i=0}^\infty \frac{1}{2^is_i}\sum_{n=0}^\infty \frac{\rho(i,n)}{\rho(i+1,n)}
\\
&= \sum_{i=0}^\infty \frac{1}{2^is_i}s_i
\\
&= \sum_{i=0}^\infty 2^{-i} 
= 2<\infty.
\end{split}\]
This concludes the proof.
\end{proof}

\begin{remark}
By \cite[Corollary 5.22.]{fremod} and \cite[Lemma 4.18]{fremod}, the full subcategory of nuclear objects in $\Ind(\Ban_R)$ is closed under countable products, arbitrary coproducts and finite tensor products. In particular, the bornological $R$-modules
\begin{equation*}
\coprod_{n\in\N} R\,e_n
\quad \text{and} \quad
\prod_{n\in\N} R\,e_n
\end{equation*}
are nuclear. 
\end{remark}

%% file: sections/duality.tex
In the previous section, we introduced a family of reflexive bornological modules. We now define algebraic structures on them. 

\begin{definition}
Let $\Comm(\Ind(\Ban_R))$ be the monoidal category of commutative monoids in $\Ind(\Ban_R)$ with the cocartesian monoidal structure induced by $\wot$. Let $\wot^{\op}$ be the cartesian monoidal structure induced by $\wot$ on the opposite category $\Comm(\Ind(\Ban_R))^\op$. A (commutative) \emph{bialgebra} over $R$ is a (commutative) monoid in $(\Comm(\Ind(\Ban_R))^\op,\wot^{\op})$, and a (commutative) \emph{Hopf algebra} over $R$ is a (commutative) group-object in $(\Comm(\Ind(\Ban_R))^\op,\wot^{\op})$.
\end{definition}

A bialgebra consists of an ind-Banach module $A$ with four operations
\begin{itemize}
    \item[-] \emph{multiplication} $m_A:A\wot A \to A$,
    \item[-] \emph{unit} $u_A:R \to A$,
    \item[-] \emph{comultiplication} $\Delta_A: A \to A\wot A$,
    \item[-] \emph{counit} $\varepsilon_A: A \to R$,
\end{itemize}
such that $\Delta_A$ and $\varepsilon_A$ are morphisms of monoids or, equivalently, $m_A$ and $u_A$ are morphisms of comonoids (i.e., monoids in the opposite category of ind-Banach modules). When $A$ is a Hopf algebra, there is also a morphism
\[
\alpha_A: A \to A
\]
which has the role of the inversion of the corresponding group-object in $\Aff_R$ and is called the \emph{antipode}.

Recall that for two ind-Banach modules $M,N$ there is a natural morphism 
\[
M^\vee \wot N^\vee \to (M\wot N)^\vee
\]
corresponding $\mathrm{ev}_M\wot \mathrm{ev}_N$ under the identifications
\[
\Hom_R\big( M^\vee \wot N^\vee, (M\wot N)^\vee\big)
\cong 
\Hom_R\big( M^\vee \wot N^\vee \wot M \wot N , R \big)
\cong
\Hom_R\big( M^\vee \wot M \wot N^\vee \wot N, R \big).
\]
In the same way we have a natural morphism
\[
M_1^\vee \wot \cdots \wot M_n^\vee 
\to 
\big( M_1 \wot \cdots \wot  M_n \big)^\vee
\]
for every finite family of ind-Banach modules $M_1,\dots,M_n$. When the above natural morphism is an isomorphism we'll write simply 
\[
M_1^\vee \wot \cdots \wot M_n^\vee 
\cong 
\big( M_1 \wot \cdots \wot  M_n \big)^\vee
\]

\begin{lemma}\label{lem:cond_dual-tensor}
Let $A$ be a bialgebra over $R$ (resp. a Hopf algebra). Suppose that the following holds for $A$: for every $n\in\N$, 
\begin{equation}\label{eq:cond_dual-tensors}
\underbrace{A^\vee \wot \cdots \wot A^\vee}_{\text{$n$ \rm times}}
\cong
\big( \underbrace{A \wot \cdots \wot A}_{\text{$n$ \rm times}} \big)^\vee.
\end{equation}
Then $A^\vee$ is a bialgebra (resp. Hopf algebra) with the structure morphisms given by applying $(-)^\vee$ and the isomorphism in \eqref{eq:cond_dual-tensors} to the structure morphisms of $A$. Moreover, if $A$ is reflexive as an ind-Banach $R$-module and $A^\vee$ satisfies condition \eqref{eq:cond_dual-tensors}, then
\[
A\cong A^{\vee\vee}
\]
canonically as bialgebras (resp. Hopf algebras).
\end{lemma}

\begin{proof}
Let $\mathrm{Span}(A,\wot)$ be the full-subcategory in $\Ind(\Ban_R)$ which consist of $R,A,A\wot A,A\wot A\wot A$ and so on. The condition \eqref{eq:cond_dual-tensors} ensures that the duality-functor $(-)^\vee$ is well-defined as a monoidal functor
\[
(-)^\vee:\mathrm{Span}(A,\wot)^{\rm op} \to \mathrm{Span}(A^\vee,\wot)
\]
The rest follows from the fact the structure of bialgebra or Hopf algebra can be entirely described by diagrams in $\mathrm{Span}(A,\wot)$ and that the shape of those diagrams in $\mathrm{Span}(A,\wot)^{\rm op}$ is again the shape of a structure of bialgebra with the role of multiplication (resp. unit) and comultiplication (resp. counit) interchanged. In the case of Hopf algebras, the antipodal morphism in $\mathrm{Span}(A,\wot)$ is also the antipodal morphism in $\mathrm{Span}(A,\wot)^{\rm op}$.
\end{proof}

\begin{definition}
A bialgebra $A$ over $R$ satisfying condition \eqref{eq:cond_dual-tensors} in Lemma \ref{lem:cond_dual-tensor} is called \emph{dualizable}. If $A$ is dualizable, the bialgebra $A^\vee$ is called the \emph{Cartier dual} or simply the \emph{dual} bialgebra of $A$. We say that two bialgebras $A,B$ form a \emph{dual pair} if they are dualizable and their underlying ind-Banach modules are dual to each other in such a way that the induced isomorphisms
\[
B\cong A^\vee,\quad A\cong B^\vee
\]
are isomorphisms of bialgebras.
\end{definition}

\subsection{Bialgebras of formal and analytic power series}

\begin{definition}\label{def:bialgebra_structure_of_formal_power_series}
Consider $\{s^0,s^1,s^2,\dots\}$ as a countable set of symbols. Define
\begin{align*}
R[s]&:=\coprod_{n\in\N} R\,s^n = \bigoplus_{n\in\N} R\,s^n
\\
R[[s]]&:=\prod_{n\in\N} R\,s^n.
\end{align*}
On $R[[s]]$ define the $R$-algebra structure given by the multiplication
\begin{equation*}
m : R[[s]]\wot R[[s]] \to R[[s]],
\quad
s^n\otimes s^k
\mapsto
s^{n+k},
\quad
n,k\in\N
\end{equation*}
and unit
\begin{equation*}
u : R \to R[[s]],
\quad
a \mapsto a s^0.
\end{equation*}
This is the ring of formal power series regarded as a bornological $R$-algebra. The same formulas define the unit and the multiplication polynomials on $R[s]$. On $R[[s]]$ we define a bi-algebra structure given by the comultiplication
\begin{equation*}
\Delta : R[[s]] \to R[[s]] \wot R[[s]],
\quad
s^n
\mapsto
(s\otimes 1 + 1\otimes s + s\otimes s)^n
=
\sum_{\substack{i,j,k\in\N\\ i+j+k=n}}\binom{n}{i,j,k} s^{i+k}\otimes s^{j+k}
\end{equation*}
and counit
\begin{equation*}
\varepsilon : R[[s]] \to R,
\quad
\sum_{n=0}^\infty a_n s^n 
\mapsto
a_0.
\end{equation*}
Here $\binom{n}{i,j,k}$ is the trinomial coefficient. It is given in terms of factorials by the formula
\[
\binom{n}{i,j,k}=\frac{n!}{i!j!k!},
\]
and it is the integer coefficient multiplying $X^iY^jZ^k$ in the polynomial expansion of $(X+Y+Z)^n$ inside the ring of polynomials $\Z[X,Y,Z]$. The same formulas define a bi-algebra structure also on $R[s]$. 
\end{definition}

Note that the morphisms in Definition \ref{def:bialgebra_structure_of_formal_power_series} are well-defined because 
\[
R[[s]]\wot R[[s]] \cong \prod_{(n,k)\in\N^2} R\,s^n\otimes s^k
\]
and, for morphisms sets of countable products of copies of $R$, there are natural isomorphisms
\[
\Hom_R\Big(\prod_{i\in\N} Re_i , \prod_{j\in\N} Re_j \Big)
\cong
 \prod_{j\in\N} \Hom_R\Big(\prod_{i\in\N} Re_i , Re_j \Big)
\cong 
\prod_{j\in\N} \coprod_{i\in\N}\Hom_R(Re_i,Re_j)
\]
by Corollary \ref{cor:duality_prod_coprod}.

\begin{lemma}
Let $A$ be a bialgebra over $R$. Suppose that the underlying ind-Banach module of $A$ is isomorphic to:
\begin{enumerate}[\rm i)]
    \item a countable coproduct of copies of $R$, or
    \item a countable product of copies of $R$.
\end{enumerate}
Then $A$ is dualizable and its Cartier dual has an underlying ind-Banach module isomorphic to one of the bornological modules above. Moreover, the Cartier dual of $A^\vee$ is canonically isomorphic to $A$.
In particular, the bialgebras $R[[s]]$ and $R[s]$ are reflexive, dualizable with dualizable Cartier dual.
\end{lemma}

\begin{proof}
Suppose that $A$ is of type i), namely 
\[
A\cong\coprod_{n\in\N} R\,e_n
\]
in $\Ind(\Ban_R)$. The tensor product commutes with coproducts and $(-)^\vee$ sends coproducts to products. The dual $A^\vee$ is a countable product of copies of $R$ and in particular it is metrizable. The tensor product $-\wot A^\vee$ commutes with countable products. Thus we get
\[\begin{split}
\underbrace{A^\vee \wot \cdots \wot A^\vee}_{k\text{ times}}
&\cong
\prod_{n_1,\dots,n_k\in\N} 
R\,\check{e}_{n_1} \otimes\cdots \otimes \check{e}_{n_k}
\\ &\cong
\left(
\coprod_{n_1,\dots,n_k\in\N} R\,e_{n_1} \otimes\cdots \otimes e_{n_k}
\right)^\vee
\\ &\cong
\big(
\underbrace{A \wot \cdots \wot A}_{k\text{ times}}
\big)^\vee
\end{split}\]
for all $k\in\Np$. This shows that if $A$ is of type i) then it is dualizable.
Suppose that $A$ is of type ii), namely 
\[
A\cong\prod_{n\in\N} R\,e_n
\]
in $\Ind(\Ban_R)$. Then $A$ is metrizable, so the functor $A\wot-$ commutes  with countable product. The dual $A^\vee$ is a countable coproduct of copies of $R$ by Corollary \ref{cor:duality_prod_coprod}. Thus we get
\[\begin{split}
\underbrace{A^\vee \wot \cdots \wot A^\vee}_{k\text{ times}}
&\cong
\coprod_{n_1,\dots,n_k\in\N} 
R\,\check{e}_{n_1} \otimes\cdots \otimes \check{e}_{n_k}
\\ &\cong
\left(
\prod_{n_1,\dots,n_k\in\N} R\,e_{n_1} \otimes\cdots \otimes e_{n_k}
\right)^\vee
\\ &\cong
\big(
\underbrace{A \wot \cdots \wot A}_{k\text{ times}}
\big)^\vee
\end{split}\]
for all $k\in\Np$.
This shows that if $A$ is of type i), then it is dualizable.
Note that $(-)^\vee$ interchanges type i) and type ii). Therefore, if $A$ is of type i) or ii), then both $A$ and $A^\vee$ are dualizable, and $(A,A^\vee)$ forms a Cartier dual pair.   
\end{proof}

\begin{definition}
Consider $R[[s]]$ and $R[s]$ with the bi-algebra structure defined in \ref{def:bialgebra_structure_of_formal_power_series}. Denote the Cartier  dual bialgebras    
\begin{align*}
\Int_R[x]&:=R[[s]]^\vee, \\
\Int_R[[x]]&:= R[s]^\vee.
\end{align*}
\end{definition}

\begin{prop}\label{pr:formal_bs-ps-duality}
As bornological modules,
\begin{gather*}
 \Int_R[x]\cong\coprod_{n\in\N} R\,\binom{x}{n}
 =\bigoplus_{n\in\N} R\,\binom{x}{n},\\
 \Int_R[[x]]\cong\prod_{n\in\N} R\,\binom{x}{n},
\end{gather*}
where $\{\binom{x}{n}:n\in\N\}$ is the dual basis of $\{s^n:n\in\N\}$, namely
\begin{equation*}
\big\langle \binom{x}{n}, s^k \big\rangle
=\begin{cases}
1 & \text{if }n=k;\\
0 & \text{if }n\neq k.
\end{cases}
\end{equation*}
The structure of bialgebra $(\Delta^\vee,\varepsilon^\vee,m^\vee,u^\vee)$ is given on the basis $\{\binom{x}{n}:n\in\N\}$ by the following formulas
\begin{equation}\label{eq:formulas_Int}
\begin{aligned}
\Delta^\vee\big(\,\binom{x}{n}\otimes\binom{x}{k}\,\big)
&=\sum_{\max\{n,k\}\leq l \leq n+k } \binom{l}{n}\binom{n}{l-k} \binom{x}{l}
\\
m^\vee\big( \binom{x}{n} \big)
&=\sum_{\substack{i,j\in\N\\ i+j=n}} \binom{x}{i}\otimes\binom{x}{j}
\\
\varepsilon^\vee(1)
&=\binom{x}{0}
\\
u^\vee\big(\, \binom{x}{n} \,\big)
&=\begin{cases}
1 & \text{if }n=0;\\
0 & \text{else.}
\end{cases}
\end{aligned}
\end{equation}
\end{prop}

\begin{proof}
We know that $(-)^\vee$ exchange countable products of copies of $R$ with countable coproducts of copies of $R$ and vice versa (Corollary \ref{cor:duality_prod_coprod}). Thus, the first part of the proposition is true, i.e.
\begin{gather*}
 \Int_R[x]\cong\coprod_{n\in\N} R\,\binom{x}{n},\\
 \Int_R[[x]]\cong\prod_{n\in\N} R\,\binom{x}{n},
\end{gather*}
where $\{\binom{x}{n}:n\in\N\}$ is dual to $\{s^n:n\in\N\}$.
We need to check that the formulas in \eqref{eq:formulas_Int} hold for the dual basis $\{\binom{x}{n}:n\in\N\}$. For $\Delta^\vee$ we can write
\begin{align*}
\Delta^\vee\big(\,\binom{x}{n}\otimes\binom{x}{k}\,\big)
&=
\sum_{l\in\N}c_l(n,k)\binom{x}{l}
\\
c_l(n,k)&=
\left\langle
\Delta^\vee\big(\,\binom{x}{n}\otimes\binom{x}{k}\,\big) , s^l
\right\rangle   
\end{align*}
for all $n,k\in\N$. Then we compute the coefficient $c_l(n,k)$:
\begin{equation*}
\begin{split}
\left\langle
\Delta^\vee\big(\,\binom{x}{n}\otimes\binom{x}{k}\,\big) , s^l
\right\rangle
&=
\left\langle
\binom{x}{n}\otimes\binom{x}{k} , \Delta(s^l)
\right\rangle
\\
&=
\left\langle
\binom{x}{n}\otimes\binom{x}{k} , 
\sum_{i,j=0}^l \binom{l}{i,j,l-i-j} s^{l-j}\otimes s^{l-i}
\right\rangle
\\
&=
\sum_{
\substack{0\leq i,j\leq l \\
l-i=n \\
l-j=k  }
} \binom{l}{i,j,l-i-j} 
\\
&=
\begin{cases}
\binom{l}{l-n,l-k,n+k-l}  & \text{if } \max\{n,k\}\leq l \leq n+k,\\
0 & \text{else}.
\end{cases}
\end{split}
\end{equation*}
If we compute $\binom{l}{l-n,l-k,n+k-l}$ in terms of factorias we get
\[
\binom{l}{l-n,l-k,n+k-l}
=
\frac{l!}{(l-n)!(l-k)!(n+k-l)!}
=\frac{l!}{(l-n)!n!}\cdot\frac{n!}{(l-k)!(n+k-l)!}.
\]
This shows that $\Delta^\vee$ satisfies the formula \eqref{eq:formulas_Int}.
For the comultiplication $m^\vee$ we have
\begin{align*}
m^\vee\big(\, \binom{x}{n} \,\big)
&=
\sum_{i,j\in\N}c_{i,j}(n)\binom{x}{i}\otimes\binom{x}{j}
\\
c_{i,j}(n)&=\left\langle 
m^\vee\big(\,\binom{x}{n}\,\big)\,,\,
s^i\otimes s^j 
\right\rangle
\end{align*}
If we compute $c_{i,j}(n)$ we get
\[
c_{i,j}(n)
=
\Big\langle 
\binom{x}{n}\,,\,
s^{i+j}
\Big\rangle
\]
which is equal to $1$ for $i+j=n$ and $0$ otherwise by definition of dual basis.
For the unit $\varepsilon^\vee(1_R)$ note that the morphism $\varepsilon:R[[s]]\to R$ is by definition the basis element $\binom{x}{0}$ because it sends $s^0$ to $1$ and every $s^n$ for $n>0$ to $0$. For the counit of $\Ind_R[x]$ we have by definition
\[
u^\vee(\binom{x}{n})=\langle \binom{x}{n}\,,\,s^0\rangle
\]
which is equivalent to the formula in \eqref{eq:formulas_Int}.
\end{proof}
The notation $\Int_R[x]$ is motivated by the fact that, if we regard $\binom{x}{n}$ as the polynomial
\[
\frac{x(x-1)(x-2)\cdots(x-n+1)}{n!},
\]
the module
\[
\coprod_{n\in\N}\Z\binom{x}{n}
\]
can be identified with the sub-ring of $\Q[x]$ consisting of the polynomial with rational coefficients that are integer valued on $\N$. Moreover, the formulas in \eqref{eq:formulas_Int} are consistent with the classical formulas for computing the binomial expansion of
\[
\binom{x}{n}\binom{x}{k} 
\quad\text{and}\quad
\binom{x+y}{n}
\]
in $\Q[x]$ and $\Q[x,y]$ respectively.

Now we consider finer analytic structures on the bi-algebras $\Int_R[x]$ and $R[[s]]$. 
\begin{definition}
For a real number $\rho>0$ define
\begin{align*}
\ps{R}{s}{\rho}&:=\coprodc_{n\in\N} [R s^n]_{\rho^n}
\\
\bs{R}{x}{\rho}&:=\prodc_{n\in\N}\Big[\,R\binom{x}{n}\,\Big]_{\rho^n}
\end{align*}
\end{definition}
The morphism $\Delta$ and $\Delta^\vee$ of $R[s]$ and $\Int_R[x]$ respectively do not extend to the $R$-modules $\ps{R}{s}{\rho}$ and $\bs{R}{x}{\rho}$ in general. On the contrary, the multiplication $m$ of $R[s]$ does extend to $\ps{R}{s}{\rho}$ because 
\[
\ps{R}{s}{\rho} \wot \ps{R}{s}{\rho}
\cong 
\coprodc_{(i,j)\in\N^2}[R\, s^i\otimes s^j]_{\rho^{i+j}}.
\]
We can compute the norm of the morphism 
\[
m:\ps{R}{s}{\rho} \wot \ps{R}{s}{\rho} \to \ps{R}{s}{\rho},
\quad
s^i \otimes s^j \mapsto s^{i+j}
\]
as the supremum
\[
\sup_{(i,j)\in\N^2} \frac{\norm{s^{i+j}}}{\norm{s^i\otimes s^j}} =
\sup_{(i,j)\in\N^2} \frac{\rho^{i+j}}{\rho^i\rho^j} =
1
\]
The unit $u:R\to \ps{R}{s}{\rho}$ is also well-defined and of norm $1$, so that $\ps{R}{s}{\rho}$ is a commutative monoid in $(\Banc_R,\wot)$.
\begin{lemma}\label{lem:points_of_the_closed_disk}
Let $A$ be any Banach $R$-algebra. Then there is a natural bijection of sets
\begin{equation}\label{eq:points_of_the_closed_disk}
\Alg_R(\ps{R}{s}{\rho},A)
\cong
\left\{
a\in A : 
\sup_{n\in\N}\ \frac{\norm{a^n}_A}{\rho^n}<\infty
\right\},
\end{equation}
where $\Alg_R(\ps{R}{s}{\rho},A)$ is the set of morphisms of monoids in $\Ban_R$. In particular, for every $a\in A$ with $\norm{a}_A\leq \rho$ there is a unique morphism $\varphi:\ps{R}{s}{\rho}\to A$ such that $\varphi(s)=a$.
If $R$ is non-archimedean and $A$ is a non-archimedean $R$-algebra then the same conclusions hold in $\Ban_R^\na$.
\end{lemma}

\begin{proof}
The set $\Alg_R(\ps{R}{s}{\rho},A)$ is a subset of the abelian group $\Hom_R(\ps{R}{s}{\rho},A)$. The latter is the underlying abelian group of the Banach $R$-module $\ihom_R(\ps{R}{s}{\rho},A)$ naturally isomorphic to
\[
\prodc_{n\in\N}[A]_{\rho^{-n}}
\]
in both $\Ban_R$ and $\Ban_R^\na$ when $R$ is non-archimedean.
Under this identification, the set $\Alg_R(\ps{R}{s}{\rho},A)$ corresponds to  the subset of sequences of the form $(1,a,a^2,a^3,\dots)$ such that 
\[
\sup_{n\in\N}\frac{\norm{a^n}_A}{\rho^n}<\infty.
\]
If $a\in A$ has norm $\norm{a}_A\leq \rho$ then 
\[
\norm{a^n}_A\leq\norm{a}_A^n\leq\rho^n.
\]
Therefore, $a$ satisfies the condition in \eqref{eq:points_of_the_closed_disk} and it corresponds to the morphism
\[
\ps{R}{s}{\rho} \to A, \quad
f(s) \mapsto f(a).
\]
\end{proof}

\begin{lemma}\label{lem:radius_for_def_of_bi-algebra_structure}
For any real number $\rho>0$, the comultiplication $\Delta:R[s]\to R[s]\wot R[s]$ extends uniquely to a well-defined morphism
\begin{equation*}
\Delta:\ps{R}{s}{2\rho+\rho^2} \to \ps{R}{s}{\rho} \wot \ps{R}{s}{\rho}.
\end{equation*}
The multiplication $\Delta^\vee:\Int_R[x] \wot \Int_R[x] \to \Int_R[x]$ extends uniquely to a well-defined morphism
\[
\Delta^\vee: \bs{R}{x}{\rho} \wot \bs{R}{x}{\rho} \to 
\bs{R}{x}{\frac{\rho^2}{2\rho +1}} 
\]
If $R$ is a non-archimedean Banach ring and $\rho\leq1$ then 
\begin{align*}
&\Delta : \ps{R}{s}{\rho} \to \ps{R}{s}{\rho} \wot \ps{R}{s}{\rho},
\\
&\Delta^\vee : \bs{R}{x}{\rho^{-1}} \wot \bs{R}{x}{\rho^{-1}} \to \bs{R}{x}{\rho^{-1}} 
\end{align*}
are well-defined in $\Ban_R^{\na}$.
\end{lemma}

\begin{proof}
Since $\bs{R}{x}{\rho}\cong\big(\ps{R}{s}{\rho^{-1}}\big)^\vee$, the statements about $\Delta^\vee$ are dual to the ones about $\Delta$. This, it is enough to prove the statements about $\Delta$.
By Lemma \ref{lem:points_of_the_closed_disk} the problem is reduced to computing the norm of 
\[
\Delta(s)=s\otimes 1 + 1\otimes s + s\otimes s.
\]
in $\ps{R}{s}{\rho}\wot \ps{R}{s}{\rho}$. The norm of the elements $s^i\otimes s^j$ is $\rho^{i+j}$ for all $i,j\in\N$. Thus
\[\begin{split}
\norm{s\otimes 1 + 1\otimes s + s\otimes s}
&\leq
\norm{s\otimes 1} + \norm{1\otimes s} + \norm{s\otimes s}
\\
&\leq
2\rho + \rho^2
\end{split}\]
and the last part of Lemma \ref{lem:points_of_the_closed_disk} applies. Suppose now that $\rho\leq1$ and $R$ is non-archimedean. Then 
\[\begin{split}
\norm{s\otimes 1 + 1\otimes s + s\otimes s}
&\leq
\sup\big\{\norm{s\otimes 1} , \norm{1\otimes s} , \norm{s\otimes s} \big\}
\\
&=
\sup\{\rho,\rho^2\}
\\
&=
\rho
\end{split}\]
By Lemma \ref{lem:points_of_the_closed_disk} we can conclude that $\Delta$ is a well-defined endomorphism of the non-archimedean $R$-algebra $\ps{R}{s}{\rho}$ 
\end{proof}
For every couple of real numbers $\rho_1,\rho_2$ satisfying the inequality $\rho_1<\rho_2$, there are canonical morphisms
\begin{align*}
& \ps{R}{s}{\rho_2} \to \ps{R}{s}{\rho_1} \\
& \bs{R}{x}{\rho_2} \to \bs{R}{x}{\rho_1}
\end{align*}
which are the identity on elements.
\begin{definition}
For $\sigma\in(0,+\infty]$ define the bornological $R$-module 
\begin{equation*}
\sps{R}{s}{\sigma}:=\varprojlim_{\rho<\sigma}\ps{R}{s}{\rho},
\end{equation*}
where the limit is taken over the open interval $(0,\sigma)$ regarded as a directed set with its natural order relation.
For $\sigma\in[0,+\infty)$, define
\begin{equation*}
\dbs{R}{x}{\sigma}:=\varinjlim_{\rho>\sigma} \bs{R}{x}{\rho},
\end{equation*}
where the colimit is taken over the open interval $(\sigma,+\infty)$ regarded as a directed set with the reversed order. We'll use the convention that $\sigma^{-1}=0$ if $\sigma=+\infty$.
\end{definition}

\begin{theorem}\label{th:duality_of_the_bialgebras}
Let $\sigma\in(0,\infty]$. Assume one of the following two conditions:
\begin{enumerate}[\rm i)]
    \item $\sigma=\infty$, or 
    \item $R$ is non-archimedean and $0<\sigma\leq1$.
\end{enumerate}
Then, the formulas given in Definition \ref{def:bialgebra_structure_of_formal_power_series} for $(m,u,\Delta,\varepsilon)$ define a bialgebra structure on the bornological $R$-algebra $\sps{R}{s}{\sigma}$. Dually, the formulas valid for $(\Delta^\vee,\varepsilon^\vee,m^\vee,u^\vee)$ define a bialgebra structure on $\dbs{R}{x}{\sigma^{-1}}$. The two bialgebras $\sps{R}{s}{\sigma}$ and $\dbs{R}{x}{\sigma^{-1}}$ form a dual pair. The duality is given by the isomorphisms
\begin{align*}
&\big( \sps{R}{s}{\sigma} \big)^\vee \to \dbs{R}{x}{\sigma^{-1}},
\quad \xi \mapsto \sum_{n=0}^\infty \langle\xi,s^n\rangle\binom{x}{n}
\\
& \big( \dbs{R}{x}{\sigma^{-1}} \big)^\vee \to \sps{R}{s}{\sigma},
\quad \mu \mapsto 
\sum_{n=0}^\infty \big\langle\mu,{\textstyle\binom{x}{n}}\big\rangle \, s^n
\end{align*}
\end{theorem}

\begin{proof}
The bornological $R$-module $\sps{R}{s}{\sigma}$ is the limit
\[
\varprojlim_{\rho\in(0,\sigma)} \ps{R}{s}{\rho}.  
\]
For any strictly increasing sequence of positive real numbers $(\rho_j)_{j\in\N}$ converging to $\sigma$ from below we have
\[
\sps{R}{s}{\sigma}=\varprojlim_{j\in\N}\ps{R}{s}{\rho_j}.
\]
Thus $\sps{R}{s}{\sigma}$ is identified with the sequential module $\lambda(\rho)$ for a function ${\rho:\N\times\N \to \Rp}$ defined by
\[
\rho(j,n)=\rho_j^n,
\quad
j,n\in\N.
\]
Analogously,
\[
\dbs{R}{x}{\sigma^{-1}}=\varinjlim_{j\in\N} \bs{R}{x}{\rho_j^{-1}}
\]
Note that $\rho$ satisfies the hypothesis of Lemma \ref{lem:prodc_coprodc_exchange_in_lambda-kappa} because
\begin{equation}\label{eq:radius_cond_main_theorem}
\sum_{n=0}^\infty \frac{\rho_j^n}{\rho_{j+1}^n}
=
\frac{1}{1-{\frac{\rho_j}{\rho_{j+1}}}}
<\infty
\end{equation}
Therefore $\kappa(\rho)$ can be equivalently computed by the colimit of products and we have the chain of isomorphism
\[\begin{split}
\kappa(\rho)
&\cong
\varinjlim_{j\in\N} \prodc_{n\in\N} [Re_n]_{\rho_j^{-n}}
\\ 
&\cong 
\varinjlim_{j\in\N} \prodc_{n\in\N} \left[R\binom{x}{n}\right]_{\rho_j^{-n}}
\\
&\cong
\dbs{R}{x}{\sigma^{-1}}
\end{split}\]
sending the basis element $e_n$ to $\binom{x}{n}$ for every $n\in\N$.
Theorem \ref{th:duality_lambda-kappa} applies to the modules $\kappa(\rho)$ and $\lambda(\rho)$ of this case thanks to the condition highlighted in Equation \eqref{eq:radius_cond_main_theorem}. Under the identification
\begin{align*}
&\kappa(\rho)\cong\dbs{R}{x}{\sigma^{-1}} \\
&\lambda(\rho)\cong\sps{R}{s}{\sigma}
\end{align*} 
the duality-pairing $\lambda(\rho)\wot\kappa(\rho)\to R$ induces the duality-isomorphisms given in the statement. This ensures that $\sps{R}{s}{\sigma}$ and $\dbs{R}{x}{\sigma^{-1}}$ are in duality as bornological $R$-modules. For the bialgebras structures we have to check only that $\Delta$ and $\Delta^\vee$ are well-defined on $\sps{R}{s}{\sigma}$ and $\dbs{R}{x}{\sigma^{-1}}$ and that the bialgebras obtained are dualizable. They will be automatically in duality by construction. 

Note that $\sps{R}{s}{\sigma}$ is a nuclear Fréchet $R$-module by Theorem \ref{th:nuclearity_kappa-lambda}, in particular it is flat (Proposition \ref{pr:nuclears_are_flat}). It also is a limit of flat projective Banach $R$-modules. Thus, by Proposition \ref{pr:lim-wot-compatibility_for_metrizables}, the tensor product commutes with the limit over $\N$:
\begin{equation*}
\underbrace{ 
\sps{R}{s}{\sigma} \wot \cdots \wot \sps{R}{s}{\sigma} 
}_{k\text{ times}}
\cong
\varprojlim_{j_1\in\N} \cdots \varprojlim_{j_k\in\N}
\ps{R}{s}{\rho_{j_1}}\wot \cdots \wot \ps{R}{s}{\rho_{j_k}}.
\end{equation*}
The $k$ projective limits are equivalent to the limit over $k$-tuples $(j_1.\dots, j_k)$ of natural numbers. Since $(j,j,\dots,j)_{j\in\N}$ is a cofinal sequence in $\N^k$, the limit over $\N^k$ is the same as the following limit over $\N$:
\[
\varprojlim_{j\in\N} 
\underbrace{
\ps{R}{s}{\rho_j} \wot \cdots \wot \ps{R}{s}{\rho_j} 
}_{k\text{ times}}
\]
From these considerations about the limit and the tensor product we can conclude that $\sps{R}{s}{\sigma}$ is dualizable through the following computations.
\begin{align*}
\big(\underbrace{ 
\sps{R}{s}{\sigma} \wot \cdots \wot \sps{R}{s}{\sigma} 
}_{k\text{ times}} \big)^\vee
&\cong
\Big(
\varprojlim_{j\in\N} 
\underbrace{
\ps{R}{s}{\rho_j} \wot \cdots \wot \ps{R}{s}{\rho_j} 
}_{k\text{ times}}
\Big)^\vee
\\ &\cong
\left(
\varprojlim_{j\in\N}\quad \coprodc_{n_1,\dots,n_k}
[R\,s^{n_1}\otimes\cdots\otimes s^{n_k}]_{\rho_j^{n_1+\cdots + n_k}}
\right)^\vee
\\ &\cong
\lambda\big( \rho_j^{d(n)_1+\dots+d(n)_k}:j,n\in\N \big)^\vee
\end{align*}
where $d:\N \to \N^k,\ n\mapsto (d(n)_1,\dots,d(n)_k)$ is a chosen bijection. Note that for every $j\in\N$
\[
\sum_{n=0}^\infty \left(\frac{\rho_j}{\rho_{j+1}}\right)^{d(n)_1+\dots+d(n)_k}
=
\sum_{n_1,\dots,n_k\in\N^k}\left(\frac{\rho_j}{\rho_{j+1}}\right)^{n_1+\cdots+n_k}
=
\left( \frac{1}{1-\rho_j/\rho_{j+1}}\right)^k.
\]
Thus Theorem \ref{th:duality_lambda-kappa} applies and we can continue the chain of isomorphism.
\begin{align*}
\lambda\big( \rho_j^{d(n)_1+\dots+d(n)_k}:j,n\in\N \big)^\vee
&\cong
\kappa\big( \rho_j^{-(d(n)_1+\dots+d(n)_k)}:j,n\in\N \big)
\\ &\cong
\varinjlim_{j\in\N}\quad
\coprodc_{n_1,\dots,n_k}
\left[
R\,\binom{x}{n_1}\otimes\cdots\otimes \binom{x}{n_k}
\right]_{\rho_j^{-(n_1+\cdots + n_k)}}
\\ &\cong
\varinjlim_{r_1,\dots, r_k>\sigma^{-1}}
\coprodc_{n\in\N} \Big[R\,\binom{x}{n}\Big]_{r_1^n}
\wot \cdots \wot
\coprodc_{n\in\N} \Big[R\,\binom{x}{n}\Big]_{r_k^n}
\\ &\cong
\underbrace{
\varinjlim_{r>\sigma^{-1}}\coprodc_{n\in\N} \Big[R\,\binom{x}{n}\Big]_{r^n}
\wot \cdots \wot
\varinjlim_{r>\sigma^{-1}}\coprodc_{n\in\N} \Big[R\,\binom{x}{n}\Big]_{r^n}
}_{k\text{ times}}
\\ &\cong
\underbrace{
\dbs{R}{x}{\sigma^{-1}} \wot \cdots \wot \dbs{R}{x}{\sigma^{-1}}
}_{k\text{ times.}}   
\end{align*}
Here we've used the fact that we can interchange contracting coproducts and contracting products in the definition of $\dbs{R}{x}{\sigma^{-1}}$ because the inductive system defining it has nuclear transition morphisms.
An analogous computation shows that also $\dbs{R}{x}{\sigma^{-1}}$ satisfies the condition for being a dualizable bialgebra.

Now we show that $\sps{R}{s}{\sigma}$ and $\dbs{R}{x}{\sigma^{-1}}$ are indeed a bialgebras. We have two cases depending on whether condition i) or ii) holds. If i) is true, by the first part of Lemma \ref{lem:radius_for_def_of_bi-algebra_structure} and the functoriality of the limit we obtain that $\Delta$ is well-defined as a morphism 
\[
\Delta : 
\varprojlim_{\rho\in(0,+\infty)} \ps{R}{s}{2\rho+\rho^2}
\to
\varprojlim_{\rho\in(0,+\infty)} \ps{R}{s}{\rho} \wot \ps{R}{s}{\rho}.
\]
The codomain is $\sps{R}{s}{+\infty} \wot \sps{R}{s}{+\infty}$ by the same argument above and the domain is $\sps{R}{s}{+\infty}$ because the function
\[
\rho \mapsto 2\rho +\rho^2
\]
is an increasing bijection of the directed set $(0,+\infty)$ with itself. 

If ii) is true then $\Delta$ is well-defined on $\ps{R}{s}{\rho}$ for any $\rho<\sigma$ by the second part of Lemma \ref{lem:radius_for_def_of_bi-algebra_structure}. By functoriality of $\varprojlim_{\rho<\sigma}$ we obtain 
\[
\Delta : 
\varprojlim_{\rho<\sigma} \ps{R}{s}{\rho}
\to
\varprojlim_{\rho<\sigma} \ps{R}{s}{\rho} \wot \ps{R}{s}{\rho}.
\]
The codomain is identified with the tensor product
\[
\sps{R}{s}{\sigma}\wot\sps{R}{s}{\sigma}
\]
as shown before. The identification above is natural and it acts on elements as the identity, so we get in the end $\Delta$ as a well-defined comultiplication on $\sps{R}{s}{\sigma}$.

The case of $\Delta^\vee$ and $\dbs{R}{x}{\sigma^{-1}}$ is the formal dual of $\Delta$ and $\sps{R}{s}{\sigma}$.
\end{proof}

\subsection{Binomial series as functions on positive integers}
For every $\sigma\in(0,\infty]$, there are well-defined morphisms
\begin{gather*}
\Int_R[x] \to \dbs{R}{x}{\sigma^{-1}} \to \Int_R[[x]] 
\\
R[s] \to \sps{R}{s}{\sigma} \to R[[s]]   
\end{gather*}
which are the identity on elements. In this way, we can identify the underlying $R$-module of $\dbs{R}{x}{\sigma^{-1}}$ (resp. $\sps{R}{s}{\sigma}$) with a submodule of $\Int_R[[x]]$ (resp. $R[[s]]$) made of formal series satisfying some constraints on the coefficients. Now we want to make the elements of $\Int_R[[x]]$ more explicit. Through a change of basis, we're going to see that the underlying set of $\Int_R[[x]]$ consists of all the functions $\N \to R$.
\begin{definition}
For every $n\in\N$, define 
\[
s_n:=(1+s)^n=\sum_{k=0}^n \binom{n}{k} s^k 
\]
in $R[s]$. Let $\beta$ be the unique endomorphism of $\coprod_{n\in\N}R\,e_n$ satisfying
\[
\beta(e_j)=\sum_{i=0}^j\binom{j}{i} e_i.
\]
The endomorphism $\beta$ is called \emph{binomial transform}.
\end{definition}
Writing
\[
\beta_{ij}=\langle \check{e}_i,\beta(e_j) \rangle=\binom{j}{i}
\]
for every $i,j\in\N$, we may regard $\beta$ as the change-of-basis matrix from the basis $\{s_n:n\in\N\}$ to the basis $\{s^n:n\in\N\}$ of $R[s]$. It is an upper triangular matrix with diagonal equal to the identity matrix:
\[
\beta \;\sim\;
\begin{pmatrix}
\binom{0}{0} & \binom{1}{0} & \binom{2}{0} & \binom{3}{0} & \binom{4}{0} & \cdots \\
0 & \binom{1}{1} & \binom{2}{1} & \binom{3}{1} & \binom{4}{1} & \cdots \\
0 & 0 & \binom{2}{2} & \binom{3}{2} & \binom{4}{2} & \cdots \\
0 & 0 & 0 & \binom{3}{3} & \binom{4}{3} & \cdots \\
0 & 0 & 0 & 0 & \binom{4}{4} & \cdots \\
\vdots & \vdots & \vdots & \vdots & \vdots & \ddots
\end{pmatrix}
=
\begin{pmatrix}
1 & 1 & 1 & 1 & 1 & \cdots \\
0 & 1 & 2 & 3 & 4 & \cdots \\
0 & 0 & 1 & 3 & 6 & \cdots \\
0 & 0 & 0 & 1 & 4 & \cdots \\
0 & 0 & 0 & 0 & 1 & \cdots \\
\vdots & \vdots & \vdots & \vdots & \vdots & \ddots
\end{pmatrix}
\]
It is easy to compute the inverse matrix of $\beta$ which gives the change of basis in the other direction:
\[
s^n=\sum_{k=0}^n (-1)^{n-k} \binom{n}{k}s_k.
\]
In particular, the binomial transform $\beta$ is an automorphism.
By duality, we have an isomorphism
\[
\prod_{k\in\N} R\,\check{e}_k \to \Int_R[[x]],
\quad
\check{e}_k \mapsto r_k
\]
which gives another basis $\{r_n:n\in\N\}$ for $\Int_R[[x]]$ dual to the basis $\{s_n:n\in\N\}$ of $R[s]$. The transpose of the matrixes associated with $\beta$ and $\beta^{-1}$ are the change-of-basis matrixes between the basis $\{\binom{x}{n}:n\in\N\}$ and the basis $\{r_n:n\in\N\}$. In particular,
\begin{align*}
\binom{x}{n}&=\sum_{k=n}^\infty \binom{k}{n}r_k 
\\
r_n&=\sum_{k=n}^\infty (-1)^{k-n}\binom{k}{n} \binom{x}{k}
\end{align*}
for every $n\in\N$. Having another basis for $R[s]$ and $\Int_R[[x]]$, we can describe the bialgebra structures in terms of the new basis. From the definition of the monoid structure $(m,u)$ on $R[s]$ we have
\begin{align*}
m(s_n\otimes s_k) &= s_{n+k} 
\\
u(1) &= s_0
\end{align*}
for all $n,k\in\N$. Since $(\Delta,\varepsilon)$ monoid morphisms, one has
\begin{align*}
\Delta(s_n)&=s_n\otimes s_n
\\
\varepsilon(s_n)&=1
\end{align*}
for all $n\in\N$. By duality, on $\Int_R[[x]]$ we get
\begin{align*}
\Delta^\vee(r_n\otimes r_k) &= \begin{cases}
    r_n & \text{if }n=k\\
    0 & \text{else}
\end{cases}
\\
\varepsilon^\vee(1) &= \sum_{i=0}^\infty r_i
\\
m^\vee(r_n) &= \sum_{i+j=n} r_i\otimes r_j
\\
u^\vee(r_n) &= \begin{cases}
    1 & \text{if }n=0\\
    0 & \text{else}
\end{cases}
\end{align*}
Let 
\[
f(x)=\sum_{k=0}^\infty a_k\binom{x}{k}
\]
denotes an element of $\Int_R[[x]]$. Note that for every $n\in\N$
\begin{equation}\label{eq:eval_f_on_n}
\langle f,s_n\rangle = \sum_{k=0}^n a_n\binom{n}{k}
\end{equation}
by the definition of $s_n$ in terms of the basis $\{s^k:k\in\N\}$. Observe that the right hand side of \eqref{eq:eval_f_on_n} is what we would get by evaluating $f(x)$ at $x=n$. The fact that the sum in \eqref{eq:eval_f_on_n} has only the first $n+1$ terms is consistent with that fact that $\binom{x}{k}$ vanishes on $n$ if $k>n$ when it is viewed as a polynomial with rational coefficients. If we define the value of $f(x)$ at $x=n$ by
\[
f(n):=\langle f, s_n \rangle,
\]
for every $n\in\N$, then 
\[
r_k(n)=\begin{cases}
    1 & \text{if }n=k\\
    0 & \text{else}
\end{cases}
\]
for all $n,k\in\N$. Thus, $r_k$ is the indicator function of the singleton set $\{k\}\subset \N$. We can write
\[
f(x)=\sum_{n=0}^\infty f(n) \, r_n(x)
\]
and from this we can see that the bialgebra structure of $\Int_R[[x]]$ expressed in the basis $\{r_n:n\in\N\}$ is the bialgebra structure on functions $\N\to R$ with pointwise multiplication and comultiplication given by the additive monoid $(\N,+)$. In more precise terms,
\[
\Int_R[[x]]=\prod_{n\in\N} R\,r_n
\]
as a bornological $R$-module, and for every $f,g\in\Int_R[[x]]$ we have
\begin{enumerate}[i)]
    \item $\Delta^\vee(f\otimes g)=fg$ where $(fg)(n)=f(n)g(n)$ for all $n\in\N$;
    \item $\varepsilon^\vee(1)$ is the constant function with value $1\in R$;
    \item $m^\vee(f)$ is a function $\N\times\N \to R$ defined by
    \[
    m^\vee(f)(i,j)=f(i+j)
    \]
    for all $i,j\in\N$;
    \item $u^\vee(f)=f(0)$.  
\end{enumerate}
In iii) we used the fact that
\[
\prod_{i\in\N} R\,r_i \; \wot \;  \prod_{j\in\N} R\,r_j
\; =
\prod_{(i,j)\in\N\times\N} R \, r_i\otimes r_j
\]
and the underlying set of $\prod_{(i,j)\in\N\times\N} R\,r_i\otimes r_j$ is the set of functions $\N\times\N \to R$. 

Now, given an element of $\Int_R[[x]]$ expressed as a binomial series, we reinterpret its coefficients in terms of the finite difference operator.
\begin{definition}
For every $f\in\Int_R[[x]]$ define $\fdiff{f}$ as the unique element of $\Int_R[[x]]$ such that
\[
\fdiff{f}(n)=f(n+1)-f(n)
\]
for every $n\in\N$. Extend the definition inductively by setting 
\begin{align*}
    \fdiff[0]{f}&=f,
    \\
    \fdiff[k+1]{f}&=\fdiff[k]{(\fdiff{f})},
\end{align*}
for every $k\in\N$. The assignment $f\mapsto \fdiff{f}$ defines an endomorphism of the bornological $R$-module $\Int_R[[x]]$, which we call the \emph{finite difference operator}. In general, for every $f\in\Int_R[[x]]$ and every $k\in\N$, we'll write $f(x+k)$ for the unique element of $\Int_R[[x]]$ corresponding to the function
\[
\N \to R,\quad
n \mapsto f(n+k).
\]
\end{definition}

\begin{prop}\label{pr:fdiff_and_binomial_expansion}
For every $k\in\N$ we have
\begin{equation}\label{eq:fdiff_of_binomial}
\fdiff{\binom{x}{k}}=\begin{cases}
    0 & \text{if }k=0 \\
    \binom{x}{k-1} & \text{if }k>0.
\end{cases}
\end{equation}
Moreover, for every $f\in\Int_R[[x]]$ the following formula holds:
\begin{equation}\label{eq:k-th_fdiff-formula}
    \fdiff[k]{f}(x)=\sum_{i=0}^k(-1)^{k-i} \binom{k}{i}f(x+i).
\end{equation}
In particular,
\begin{equation}\label{eq:binom_coeff_as_fdiff}
f(x)=\sum_{n=0}^\infty \fdiff[n]{f}(0)\binom{x}{n}.
\end{equation}
\end{prop}

\begin{proof}
Thanks to the fact that 
\[
\Int_R[[x]]=\prod_{k\in\N} R\,r_k,
\]
every part of the statement follows from classical computations valid as if we were considering $x$ as a variable ranging in $\N$. For example, we have that
\[
\binom{n+1}{k}-\binom{n}{k}
=
\frac{(n+1)!-(n+1-k)\,n!}{k!(n+1-k)!}
=
\binom{n}{k-1}
\]
for every $n\in\N$ and $k\in\Np$. Thus
\[\begin{split}
\fdiff{\binom{x}{k}} 
&= 
\sum_{n=0}^\infty \left[\binom{n+1}{k}-\binom{n}{k}\right]r_n(x) \\
&=
\sum_{n=0}^\infty \binom{n}{k-1} r_n(x) \\
&=
\binom{x}{k-1}.
\end{split}\]
Note that the finite difference operator is equal to $T-\id$, where $T$ is the endomorphism
\[
T:\Int_R[[x]] \to \Int_R[[x]],
\quad
f(x)\mapsto f(x+1).
\]
For every $k\in\N$, the element $\fdiff[k]{f}$ is equal to $(T-\id)^kf$.
Since $T$ and the identity of $\Int_R[[x]]$ commutes to each other, we can expand $(T-\id)^k$ and get
\[\begin{split}
(T-\id)^k=\sum_{i=0}^k\binom{k}{i} T^i\circ(-\id)^{k-i},
\end{split}\]
which applied to $f$ produces \eqref{eq:k-th_fdiff-formula}. In particular, for $x=0$ we have
\begin{equation}\label{eq:k-th_fdiff-formula_at_0}
\fdiff[k]{f}(0)=\sum_{i=0}^k(-1)^{k-i}\binom{k}{i}f(i).
\end{equation}
To check that \eqref{eq:binom_coeff_as_fdiff} holds, take any 
\[
f(x)=\sum_{n=0}^\infty a_n\binom{x}{n},
\]
in $\Int_R[[x]]$. Then $a_n = \langle f, s^n \rangle$ for every $n\in\N$. If we express $s^n$ in the basis $\{s_i:i\in\N\}$ we get
\[
a_n=\sum_{i=0}^n (-1)^{n-i}\binom{n}{i}\langle f, s_i \rangle
=
\sum_{i=0}^n (-1)^{n-i}\binom{n}{i} f(i)
\]
which is equal to $\fdiff[n]{f}(0)$ by Equation \eqref{eq:k-th_fdiff-formula_at_0}.
\end{proof}

Proposition \ref{pr:fdiff_and_binomial_expansion} gives a description of the underlying sets of the bornological modules $\Int_R[x]$ and $\dbs{R}{x}{\sigma}$ for $\sigma\in\R$ satisfying $\sigma\geq 0$. Let $f:\N \to R$ be a function. Then
\begin{enumerate}[i)]
    \item $f$ belongs to $\Int_R[x]$ if and only if there is $n_0\in\N$ such that $\fdiff[n]f(0)=0$ for every integer $n\geq n_0$;
    \item $f$ belongs to $\dbs{R}{x}{\sigma}$ if and only if there exists $\rho>\sigma$ such that
    \[
    \lim_{n\to\infty} \abs{\fdiff[n]{f}(0)} \rho^n = 0.
    \]
\end{enumerate}
\begin{remark}
The finite difference operator is a well-defined endomorphism of $\bs{R}{x}{\rho}$ for every $\rho\in\Rp$. Given $f\in\bs{R}{x}{\rho}$ as a series
\[
f(x)=\sum_{n=0}^\infty a_n \binom{x}{n},
\]
the element $\fdiff{f}$ is obtained by shifting the coefficients:
\[
\fdiff{f}(x)=\sum_{n=0}^\infty a_{n+1}\binom{x}{n}.
\]
The finite difference operator is bounded because
\[
\norm{\fdiff{f}}_{\bs{R}{x}{\rho}}\leq \rho^{-1} \norm{f}_{\bs{R}{x}{\rho}}
\]
as one can deduce from
\[
\abs{a_{n+1}}\rho^n
=
\rho^{-1}\abs{a_{n+1}}\rho^{n+1}.
\]
The endomorphism $f(x)\mapsto f(x+1)$ is also well-defined on $\bs{R}{x}{\rho}$ because it is the identity plus the finite difference operator. By the functoriality of colimits, we get that the above endomorphisms are well-defined also on $\dbs{R}{x}{\sigma}$ for every $\sigma\geq0$.
\end{remark}

\subsection{Hopf algebra structure}
We saw in Theorem \ref{th:duality_of_the_bialgebras} that when $R$ is non-archimedean, $\sps{R}{s}{1}$ and $\dbs{R}{x}{1}$ form a dual pair of bialgebras. In this case, they also have a well-defined antipodal morphism that gives them the structure of Hopf algebras. At first, we define the antipodal morphism on $R[[s]]$. Observe that every endomorphism $\alpha$ of $R[[s]]$ can be uniquely represented by an infinite square matrix in which each row has only finitely many non-zero entries.
In order for such an endomorphism to respect the algebra structure of $R[[s]]$, it is necessary and sufficient that
\[
\alpha(s^n) = \alpha(s)^n \quad \text{for all } n\in\N.
\]
These conditions imply that any monoid endomorphism $\alpha : R[[s]] \to R[[s]]$ is uniquely determined by the choice of a single power series $f(s) = \alpha(s)$ with vanishing constant term.
\begin{definition}
Define the morphism $\alpha:R[[s]]\to R[[s]]$ as the unique monoid endomorphism satisfying
\[
\alpha(s)=(1+s)^{-1}-1=\sum_{n=1}^{\infty}(-1)^{n} s^n.
\]
\end{definition}

\begin{lemma}\label{lem:formal_antipode}
The endomorphism $\alpha$ induces on the bialgebra $(R[[s]],m,u,\Delta.\varepsilon)$ the structure of a Hopf algebra.
\end{lemma}

\begin{proof}
We have to check that 
\[
m\circ(\alpha \, \wot \, \id_{R[[s]]})\circ \Delta
=
m\circ( \id_{R[[s]]}\, \wot \, \alpha )\circ \Delta
=
u\circ\varepsilon
\]
It is enough to check what the morphisms do on $s\in R[[s]]$, or equivalently on $1+s$. We get
\begin{align*}
\big( m\circ(\alpha \, \wot \, \id_{R[[s]]})\circ \Delta \big) (1+s)
&=&
m\big( (1+s)^{-1}\otimes (1+s) \big)
&=&
1
\\
\big( m\circ( \id_{R[[s]]} \, \wot \, \alpha )\circ \Delta \big) (1+s)
&=&
m\big( (1+s)\otimes (1+s)^{-1} \big)
&=&
1
\\
(u\circ\varepsilon)(1+s)
&=&
u(1)
&=&
1
\end{align*}
\end{proof}

\begin{prop}
Suppose that $R$ is an non-archimedean Banach ring. Then, the mapping 
\[
\alpha: s\longmapsto \sum_{n=1}^\infty (-1)^n s^n
\]
induces well-defined endomorphisms of $\ps{R}{s}{\rho}$ for all $\rho\in(0,1)$ and of $\sps{R}{s}{1}$. Moreover, the tuple of morphisms $(m.u,\Delta,\varepsilon,\alpha)$ induces a Hopf algebra structure on $\ps{R}{s}{\rho}$ for all $\rho\in(0,1)$ and on $\sps{R}{s}{1}$.
\end{prop}

\begin{proof}
By Lemma \ref{lem:points_of_the_closed_disk} it is enough to check that 
\[
f(s):=\sum_{n=1}^\infty (-1)^n s^n
\]
has norm less or equal than $\rho$. Since $R$ is non-archimedean and $0<\rho<1$, we have
\[
\norm{f}_{\ps{R}{s}{\rho}}
\leq
\sup_{n=1,2,3,\dots}\abs{(-1)^n}\rho^n
\leq
\rho.
\]
By functoriality of limits, $\alpha$ induces a well-defined endomorphism of $\sps{R}{s}{1}$. Lemma \ref{lem:formal_antipode} together with Lemma \ref{lem:radius_for_def_of_bi-algebra_structure} ensure that $\alpha$ defines the antipode of an Hopf algebra also on $\ps{R}{s}{\rho}$ and $\sps{R}{s}{1}$.
\end{proof}

\begin{remark}
The fact that there is a Hopf algebra structure for $R$ non-archimedean is consistent with the fact that in a non-archimedean Banach ring the open unit ball with centre in $1$ is closed under multiplication and inversion. This is in stark contrast with what happens in the archimedean case.
\end{remark}

Now that we know that for $R$ non-archimedean $\sps{R}{s}{1}$ is a Hopf algebra, we also know that its dual $\dbs{R}{x}{1}$ is a Hopf algebra with antipode given by $\alpha^\vee$. One could ask how $\alpha^\vee$ acts on elements of $\dbs{R}{x}{1}$. Let $f$ be one of them. The series $\alpha^\vee f$ is uniquely determined by its pairing with the elements $(1+s)^n$ with $n$ varying in $\N$: 
\[\begin{split}
\alpha^\vee f(n)
&=
\langle \alpha^\vee f, (1+s)^n \rangle
\\
&=
\langle  f, \alpha\big((1+s)^n\big) \rangle
\\
&=
\langle f, (1+s)^{-n} \rangle
\\
&=
\sum_{k=0}^\infty \binom{-n}{k} \fdiff[k]{f}(0)
\end{split}\]
Thus, we could write $\alpha^\vee f(x)=f(-x)$. Like polynomials, the elements of $\dbs{R}{x}{1}$ define functions with values in $R$ and the domain may depends on the coefficient Banach ring $R$. Suppose we have a non-archimedean Banach ring $A$ over $R$. Consider the set of $R$-algebra homomorphism from $\dbs{R}{x}{1}$ to $A$
\[ G(A):=\Alg_R(\dbs{R}{x}{1},A). \]
This is has an abelian group structure induced by the Hopf algebra structure of $\dbs{R}{x}{1}$. Denote by $+$ the group operation of $G(A)$. Every element $\varphi\in G(A)$ is also an element of the bornologial $A$-module 
\[
\ihom_R(\dbs{R}{x}{1}, A)
\]
The latter is identified with $\sps{A}{s}{1}$ by sending each morphism $\xi:\dbs{R}{x}{1}\to A$ to the series
\[
\sum_{n=0}^\infty \xi\binom{x}{n}\,s^n.
\]
If $\varphi\in G(A)$, denote by $(1+s)^\varphi$ the corresponding series in $\sps{A}{s}{1}$. Then the mapping $\varphi\mapsto (1+s)^\varphi$ identifies $G(A)$ with the subset of series $F(s)\in\sps{A}{s}{1}$ such that
\begin{gather*}
    F(0)=1, \\
    \Delta(F(s))=F(s)\otimes F(s).
\end{gather*}
These are called \emph{group-like elements} of $\sps{A}{s}{1}$. In general we can regard elements of $\dbs{R}{x}{1}$ as functions on $G(R)$ and the group-like elements of $\sps{R}{s}{1}$ as Dirac measures. For example, when $R=\Q_p$, the set of group-like elements of $\sps{\Q_p}{s}{1}$ includes the series of the form
\[
(1+s)^a=\sum_{n=0}^\infty \binom{a}{n} s^n
\]
for every $a\in\Z_p$. Thus, the elements of $\dbs{\Q_p}{x}{1}$ define functions on $\Z_p$ with values on $\Q_p$. 

%% file: sections/concl.tex

\begin{definition}
Let $p$ be a prime number. On the field of $p$-adic numbers $\Q_p$ let $\abs{-}_p$ be the unique norm satisfying
\[
\abs{x}_p=p^{-v},
\]
where $x\in\Q_p^\times$ and $v$ is the maximal integer such that $x\in p^v\Z_p$. Denote by $\C_p$ the completed algebraic closure of $\Q_p$ with the unique norm that extends $\abs{-}_p$.
Denote by $\Z_\na$ the non-archimedean Banach ring consisting of the integers with the trivial norm
\[
\abs{n}_0=\begin{cases}
    1 & \text{if }n\neq 0, \\
    0 & \text{if }n=0.
\end{cases}
\]
Denote by $\Q_\na$ the non-archimedean Banach ring consisting of the field of rational numbers with the norm
\[
\abs{x}_{\Q_\na}=\sup_{\substack{p=0 \text{ or,} \\ p\text{ prime}}} \abs{x}_p.
\]
\end{definition}
Note that the norms of the Banach rings $\Z_\na$ and $\Q_\na$ satisfy
\[
\abs{x}\geq 1
\]
for every $x\neq0$. In particular, the topology on $\Z_\na$ and $\Q_\na$ induced by their respective norm is discrete. The inclusion of the integers in $\Z_p$ induces, for every prime $p$, a morphism
\[
\Z_\na \to \Z_p
\]
in $\Ban_{\Z_\na}$. Analogously, we have canonical inclusion-morphisms
\[
\Q_\na \to \Q_p
\]
for every prime $p$. 

\begin{definition}
Given a prime $p$ and a closed sub-ring $K$ of $\C_p$, denote by 
$\sC(\Z_p,K)$ the Banach $K$-module of continuous functions $f:\Z_p \to K$ with the max-norm 
\[
\norm{f}_{\sC(\Z_p,K)}=\max_{x\in\Z_p} \abs{f(x)}_p.
\]
For $r\in (0,\infty)$ and $x_0\in\Z_p$, a continuous function $f:\Z_p \to K$ is said analytic in the closed disk $B(x_0,r)$ of radius $r$ and centre $x_0$ if there is a power series $F_{x_0,r}(T)=\sum_{n=0}^\infty a_n T^n$ in 
$\ps{K}{T}{r}$ such that
\[
f(x)=\sum_{n=0}^\infty a_n (x-x_0)
\]
for every $x\in B(x_0,r)$. Denote by $\sC^{\la,r}(\Z_p,K)$ the Banach $K$-module consisting of continuous functions $f:\Z_p \to K$ that are analytic on $B(x_0,r)$ for each $x_0\in\Z_p$. The norm is 
\[
\norm{f}_{\sC^{\la,r}(\Z_p,K)}
=
\sup_{x_0\in\Z_p} \norm{F_{x_0,r}(T)}_{\ps{K}{T}{r}}.
\]
Let $\sC^\la(\Z_p,K)$ be the bornological $K$-module obtained as the union of the Banach spaces $\sC^{\la,r}(\Z_p,K)$ as $r\to 0$. The elements of $\sC^\la(\Z_p,K)$ are called \emph{locally analytic functions} \cite[14]{colmez}. Let $\sD^\la(\Z_p, K)$ be the dual of $\sC^\la(\Z_p,K)$, whose elements are called \emph{locally analytic distributions}
\end{definition}

The binomial polynomials 
\[
\binom{x}{n}=\frac{1}{n!}x(x-1)\cdots(x-n+1)
\]
induce continuous functions $\Z_p\to \Q_p$. They actually have image in $\Z_p$ because $\N$ is dense in $\Z_p$ and $\abs{\binom{x}{n}}_p\leq1$ for $x\in\N$. By a result \cite[Théorème I.2.3]{colmez} due to Mahler, the mapping
\begin{equation}\label{eq:bin_expansion_of_Z_p-functions}
\sC(\Z_p,K) \to \coprodc_{n\in\N} K\,\binom{x}{n},
\quad
f \mapsto \sum_{n=0}^\infty \fdiff[n]{f}(0)\binom{x}{n}
\end{equation}
is an isometric isomorphism. The isomorphism \eqref{eq:bin_expansion_of_Z_p-functions} restricts to an isomorphism of bornological $K$-modules
\[
\sC^\la(\Z_p,K)\cong\dbs{K}{x}{1}
\]
(this is a result of Amice; look at \cite[162-163]{inter} or \cite[Corollaire I.4.8]{colmez} for reference). Given a locally analytic distribution $\mu$, the pairing with a locally analytic function $f$ is denoted by
\begin{equation*}
\int_{\Z_p} f(x)\,\mu(x).
\end{equation*}
The Amice transform of $\mu$ is the series
\begin{equation*}
    \sA_\mu(s):=\sum_{n=0}^\infty \mu_n s^n,
    \qquad \text{where}\quad\mu_n:=\int_{\Z_p}\binom{x}{n}\,\mu(x).
\end{equation*}
As summarized in \cite{duals}, the Amice transform defines an isomorphism 
\begin{equation*}
\sD^\la(\Z_p,K)\cong\sps{K}{s}{1}   
\end{equation*}
From this point of view, the construction of $\dbs{R}{x}{1}$ and 
$\sps{R}{s}{1}$ with their duality extends the Amice transform to a general base Banach ring $R$. It also shows that the dual of $\sps{K}{s}{1}$ is the module $\sC^\la(\Z_p,K)$ and that all the previous facts hold if we replace $K$ with any closed sub-ring of $\C_p$. One can also use as coefficients a base ring below $\Q_p$, e.g. $R=\Z_\na$ or $R=\Q_\na$. If one do so, each element $\sps{R}{s}{1}$ induce a distribution on $\Z_p$ for all primes $p$ in a uniform way. 
For example, the series
\begin{equation*}
    \log(1+s):=\sum_{n=1}^\infty \frac{(-1)^{n+1}}{n} s^n
\end{equation*}
belongs to $\sps{\Q_\na}{s}{1}$ because 
\[
\abs{\frac{(-1)^{n+1}}{n}}\leq n
\]
for all $n\in\Np$ and $n\rho^n$ is a summable sequence for every $\rho<1$. Then, for all primes $p$, the morphism $\Q_\na \to \Q_p$ induce a morphism
\[
\sps{\Q_\na}{s}{1} \to \sps{\Q_p}{s}{1}\cong\sD^\la(\Z_p,\Q_p)
\]
which sends $\frac{1}{s}\log(1+s)$ to the Kubota-Leopold distribution $\mu_{\mathrm{KL},p}$. This distribution plays a role in the $p$-adic interpolation theory of $L$-functions, especially because of the following property:
\begin{equation*}
\int_{\Z_p} x^n\,\mu_{\mathrm{KL},p}(x) = B_n
\end{equation*}
for all $n\in\Np$, where $B_n$ is the $n$-th Bernoulli number. For more details, see \cite[Section 1.8.3]{font}.

\subsubsection{A remark on base-change}
For any non-archimedean Banach ring morphism $R \to A$,  we can compute the base-change of $\dbs{R}{x}{1}$ to $A$ only using the fact that $A\wot_R-$ commutes with colimits in $\Ind(\Ban_R)$ and $\Banc_R$:
\begin{align*}
A\wot_R\dbs{R}{x}{1}
&\cong A \;\wot_R \varinjlim_{\rho>1} 
\coprodc_{n\in\N} \left[R\binom{x}{n} \right]_{\rho^n}
\\
&\cong \varinjlim_{\rho>1} A \wot_R
\coprodc_{n\in\N} \left[R\binom{x}{n} \right]_{\rho^n}
\\
&\cong \varinjlim_{\rho>1} 
\coprodc_{n\in\N} \left[A\binom{x}{n} \right]_{\rho^n} 
\\
&\cong \dbs{A}{x}{1}.
\end{align*}

The identification 
\begin{equation*}
A \;\wot_R \; \sps{R}{s}{1} \cong \sps{A}{s}{1} 
\end{equation*}
is obtained similarly if we write $\sps{R}{s}{1}$ as a colimit of weighted $\ell^1_*$-modules with nuclear transition morphisms (this is possible because $\sps{R}{s}{1}$ is nuclear). With the above identification in mind, it's easy to see that the base-change $A\wot_R-$ also sends the pairing  associated with $\sps{R}{s}{1}$ and $\dbs{R}{x}{1}$ to the pairing associated with $\sps{A}{s}{1}$ and $\dbs{A}{x}{1}$. Thus, when $R=\Z_{\na}$, we have a couple of Hopf algebras $\big(\sps{\Z_\na}{s}{1},\dbs{\Z_\na}{x}{1}\big)$ in duality, whose base-change to any complete sub-ring $K$ of $\C_p$ is the pair $\big(\sD^\la(\Z_p,K),\sC^\la(\Z_p,K)\big)$.

%% file: ref.bib
@article{dag,
title = {Dagger geometry as Banach algebraic geometry},
journal = {Journal of Number Theory},
volume = {162},
pages = {391-462},
year = {2016},
issn = {0022-314X},
doi = {https://doi.org/10.1016/j.jnt.2015.10.023},
author = {Federico Bambozzi and Oren Ben-Bassat}
}

@article{fremod,
  title={Fr{\'e}chet modules and descent},
  author={Oren Ben-Bassat and Kobi Kremnizer},
  journal={Theory and Applications of Categories},
  year={2023},
  url={https://api.semanticscholar.org/CorpusID:211506624}
}

@article{nucle,
title = {Nuclearity in the category of complete semilattices},
journal = {Journal of Pure and Applied Algebra},
volume = {57},
number = {1},
pages = {67-78},
year = {1989},
issn = {0022-4049},
doi = {https://doi.org/10.1016/0022-4049(89)90028-5},
author = {D.A. Higgs and K.A. Rowe}
}

@article{sheafy,
author = {Bambozzi, Federico and Kremnizer, Kobi},
title = {On the sheafyness property of spectra of Banach rings},
year = {2024},
journal = {Journal of the London Mathematical Society},
volume = {109},
number = {1},
pages = {e12855},
doi = {https://doi.org/10.1112/jlms.12855}
}

@inproceedings{sga,
author="Grothendieck, A.
and Verdier, J. L.",
title="Prefaisceaux",
booktitle="Th{\'e}orie des Topos et Cohomologie Etale des Sch{\'e}mas",
year="1972",
publisher="Springer Berlin Heidelberg",
address="Berlin, Heidelberg",
pages="1--217",
isbn="978-3-540-37549-4"
}

@unpublished{condborn,
    author = {Federico Bambozzi and Kobi Kremnizer},
    title = {Relations between bornological and condensed structures -
algebraic theory},
    note = {in preparation.}
}

@article{quab,
    author    = {Schneiders, Jean-Pierre},
    title     = {Quasi-abelian categories and sheaves},
    journal   = {M{\'e}moires de la Soci{\'e}t{\'e} Math{\'e}matique de France},
    series    = {Nouvelle s{\'e}rie},
    number    = {76},
    year      = {1999},
    pages     = {1--144},
    publisher = {Soci{\'e}t{\'e} Math{\'e}matique de France},
    doi       = {10.24033/msmf.389},
}

@Inbook{koethe,
author="K{\"o}the, Gottfried",
title="Some Special Classes of Locally Convex Spaces",
bookTitle="Topological Vector Spaces I",
year="1983",
publisher="Springer Berlin Heidelberg",
address="Berlin, Heidelberg",
pages="367--436",
isbn="978-3-642-64988-2",
doi="10.1007/978-3-642-64988-2_6",
url="https://doi.org/10.1007/978-3-642-64988-2_6"
}

@article{colmez,
  author  = {Colmez, Pierre},
  title   = {Fonctions d'une variable {$p$}-adique},
  journal = {Ast{\'e}risque},
  volume  = {330},
  year    = {2010},
  pages   = {13--59},
  note    = {Repr{\'e}sentations {$p$}-adiques de groupes {$p$}-adiques II : Repr{\'e}sentations de {$\mathbf{GL}_2(\mathbf{Q}_p)$} et {$(\varphi, \Gamma)$}-modules},
  url     = {https://www.numdam.org/item/AST_2010__330__13_0/}
}

@inproceedings{duals,
  author    = {Amice, Yvette},
  title     = {Duals},
  booktitle = {Proceedings of the Conference on {$p$}-adic Analysis (Nijmegen, 1978)},
  publisher = {Katholieke Univ., Nijmegen},
  year      = {1978},
  pages     = {1--15},
  note      = {MR 522117}
}

@article{inter,
  author  = {Amice, Yvette},
  title   = {Interpolation {$p$}-adique},
  journal = {Bulletin de la Soci{\'e}t{\'e} Math{\'e}matique de France},
  volume  = {92},
  year    = {1964},
  pages   = {117--180},
  doi     = {10.24033/bsmf.1606}
}

@misc{font,
  author       = {Colmez, Pierre},
  title        = {Fontaine's rings and $p$-adic $L$-functions},
  howpublished = {Lecture notes, Tsinghua University},
  year         = {2004},
  url          = {https://webusers.imj-prg.fr/~pierre.colmez/tsinghua.pdf}
}

@misc{condmath,
      title={Lectures on Condensed Mathematics}, 
      author={Peter Scholze},
      year={2026},
      eprint={2605.03658},
      archivePrefix={arXiv},
      primaryClass={math.NT},
      url={https://arxiv.org/abs/2605.03658}, 
}

@misc{persp,
      title={A Perspective on the Foundations of Derived Analytic Geometry}, 
      author={Oren Ben-Bassat and Jack Kelly and Kobi Kremnizer},
      year={2024},
      eprint={2405.07936},
      archivePrefix={arXiv},
      primaryClass={math.AG},
      url={https://arxiv.org/abs/2405.07936}, 
}

@misc{kelly2025,
      title={Localising invariants in derived bornological geometry}, 
      author={Jack Kelly and Devarshi Mukherjee},
      year={2025},
      eprint={2505.15750},
      archivePrefix={arXiv},
      primaryClass={math.KT},
      url={https://arxiv.org/abs/2505.15750}, 
}

@misc{claused2026,
      title={Condensed Mathematics and Complex Geometry}, 
      author={Dustin Clausen and Peter Scholze},
      year={2026},
      eprint={2605.11731},
      archivePrefix={arXiv},
      primaryClass={math.CV},
      url={https://arxiv.org/abs/2605.11731}, 
}
